\documentclass{amsart}
\usepackage{amssymb, amsmath,amsthm }

\usepackage{hyperref}
\usepackage{stmaryrd}
\usepackage[all]{xy}
\usepackage{dsfont}

\RequirePackage{mathrsfs}
\let\mathcal\mathscr

\theoremstyle{plain}

\newtheorem{prop}{Proposition}[section]

\newtheorem{lem}[prop]{Lemma}
\newtheorem{thm}[prop]{Theorem}

\newtheorem{cor}[prop]{Corollary}

\theoremstyle{remark}
\newtheorem{remar}[prop]{Remark}
\theoremstyle{definition}

\DeclareMathAlphabet{\mathpzc}{OT1}{pzc}{m}{it}

\DeclareMathOperator{\End}{End}
\DeclareMathOperator{\Hom}{Hom}

\DeclareMathOperator{\Sym}{Sym}

\DeclareMathOperator{\GL}{GL}
\DeclareMathOperator{\SL}{SL}

\DeclareMathOperator{\WD}{WD}
\DeclareMathOperator{\Gal}{Gal}

\DeclareMathOperator{\tr}{tr}

\DeclareMathOperator{\Spec}{Spec}
\DeclareMathOperator{\mSpec}{m-Spec}

\DeclareMathOperator{\Ext}{Ext}

\newcommand{\Qp}{\mathbb {Q}_p}

\newcommand{\Zp}{\mathbb{Z}_p}
\newcommand{\Qpbar}{\overline{\mathbb{Q}}_p}

\newcommand{\Eins}{\mathbf 1}

\newcommand{\ZZ}{\mathbb Z}

\newcommand{\Fp}{\mathbb F_p}

\newcommand{\mm}{\mathfrak m}

\newcommand{\st}{\mathrm{st}}

\newcommand{\OO}{\mathcal O}

\newcommand{\gal}{\mathcal G_{\Qp}}

\DeclareMathOperator{\wtimes}{\widehat{\otimes}}

\newcommand{\nn}{\mathfrak n}

\newcommand{\md}{\mathrm m}

\newcommand{\rr}{\mathfrak r}
\newcommand{\pp}{\mathfrak p}

\newcommand{\qq}{\mathfrak{q}}

\newcommand{\cont}{\mathrm{cont}}

\newcommand{\univ}{\mathrm{univ}}
\newcommand{\CH}{\mathrm{CH}}

\newcommand{\ver}{\mathrm{ver}}

\newcommand{\ps}{\mathrm{ps}}
\newcommand{\wt}{\mathbf w}
\newcommand{\irr}{\mathrm{irr}}
\newcommand{\op}{\mathrm{op}}
\DeclareMathOperator{\LLL}{LL}
\title{ On $2$-adic deformations}
\author{Vytautas Pa\v{s}k\={u}nas}
\date{\today.}

\begin{document} 

\begin{abstract} We compute the versal deformation ring of a split generic $2$-dimensional representation $\chi_1\oplus \chi_2$ of the absolute Galois group  of $\Qp$.
As an application, we show that the Breuil--M\'ezard conjecture for both non-split extensions of $\chi_1$ by $\chi_2$ and $\chi_2$ by $\chi_1$  implies the  Breuil--M\'ezard conjecture for $\chi_1\oplus \chi_2$. The result is new for $p=2$, the proof works for all primes. 
\end{abstract}

\maketitle

\section{Introduction}
Let $L$ be a finite extension of $\Qp$ with the ring of integers $\OO$ and residue field $k$.  Let $\gal$ be the absolute Galois group of $\Qp$ and let $\chi_1, \chi_2:\gal\rightarrow k^{\times}$ be continuous group homomorphisms, such that $\chi_1\chi_2^{-1}\neq \Eins, \omega^{\pm 1}$, where $\omega$ is the 
cyclotomic character modulo $p$. We let 
\begin{equation}\label{matrices}
\rho_1:= \begin{pmatrix} \chi_1 & \ast \\ 0 & \chi_2\end{pmatrix}, \quad \rho_2:= \begin{pmatrix} \chi_1 & 0 \\ \ast & \chi_2\end{pmatrix},
\end{equation}
be non-split extensions. The assumption  $\chi_1\chi_2^{-1}\neq \Eins, \omega^{\pm 1}$ implies that both groups $\Ext^1_{\gal}(\chi_1, \chi_2)$ and $\Ext^1_{\gal}(\chi_2, \chi_1)$ are $1$-dimensional, thus $\rho_1$ and $\rho_2$
are uniquely determined up to an isomorphism. 

Let $\mathfrak A$ be the category of local artinian augmented $\OO$-algebras with residue field $k$. Let $D_1$, $D_2$ be functors from $\mathfrak A$ to 
the category of sets, such that for $A\in \mathfrak A$ and $i=1,2$, $D_i(A)$ is the set of deformations of $\rho_i$ to $A$.
Since $\rho_1$ and $\rho_2$ have scalar endomorphisms the functors $D_1, D_2$ are pro-represented by the universal deformation rings $R_1$, $R_2$, respectively. We let $\rho_1^{\univ}$, $\rho_2^{\univ}$ be the universal deformations of $\rho_1$ and $\rho_2$, respectively. 

The representations $\rho_1^{\univ}$, $\rho_2^{\univ}$ are naturally pseudo-compact modules over the completed group algebra $\OO[[\gal]]$. 
The main purpose of this note is to compute the ring $\End^{\cont}_{\OO[[\gal]]}(\rho_1^{\univ}\oplus \rho_2^{\univ})$. The motivation for this 
computation stems from a companion paper \cite{p2}, which shows that the category of pseudo-compact modules over this ring is naturally 
anti-equivalent to a certain subcategory of smooth $\GL_2(\Qp)$-representations on $\OO$-torsion modules. In fact, for this application 
 we have to work with a fixed determinant, but we will ignore it in this introduction. If $p>2$ then this result has been proved in \cite[\S B.2]{cmf}.
 The proof there uses a result of B\"ockle \cite{boeckle}, which realizes $\rho_1^{\univ}$ and $\rho_2^{\univ}$ concretely by writing down 
 matrices for the topological generators. B\"ockle's paper in turn uses results of Pink \cite{pink} on the classification of pro-$p$ subgroups of $\SL_2(R)$, 
 where $R$ is a $p$-adic ring with $p>2$.
 
 In this paper, we give a different argument, which works for all primes $p$, and obtain a more intrinsic description of $\End^{\cont}_{\OO[[\gal]]}(\rho_1^{\univ}\oplus \rho_2^{\univ})$, which we will now describe. Let $D^{\ps}: \mathfrak A\rightarrow Sets$  be the functor, which sends $A$ to the set of $2$-dimensional $A$-valued determinants lifting the pair $(\chi_1+\chi_2, \chi_1 \chi_2)$. The notion of an $n$-dimensional determinant has been 
 introduced by Chenevier in \cite{che_det}. If $p>n$ then it is equivalent to that of an $n$-dimensional pseudo-representation (pseudo-character).
  The functor $D^{\ps}$ is pro-represented by a complete local noetherian $\OO$-algebra $R^{\ps}$. We let $(t^{\univ}, d^{\univ})$ be the universal object. 
 We show that  $\End^{\cont}_{\OO[[\gal]]}(\rho_1^{\univ}\oplus \rho_2^{\univ})$ is naturally isomorphic to the algebra opposite to the 
 Cayley--Hamilton algebra $\CH(R^{\ps}):= R^{\ps}[[\gal]]/J$, where $J$ is the closed two-sided ideal in $R^{\ps}[[\gal]]$ generated by the elements
 $g^2-t^{\univ}(g)g +d^{\univ}(g)$, for all $g\in \gal$. We also show that $\CH(R^{\ps})$ is a free $R^{\ps}$-module of rank $4$ and compute the
 multiplication table for the generators, see Proposition \ref{relation}. From this description we deduce that the centre of $\End^{\cont}_{\OO[[\gal]]}(\rho_1^{\univ}\oplus \rho_2^{\univ})$ is naturally isomorphic to $R^{\ps}$. As a part of the proof we show in Proposition \ref{trace_rep} that mapping a representation to its trace and determinant induces isomorphisms $R^{\ps}\overset{\cong}{\rightarrow} R_1$, $R^{\ps}\overset{\cong}{\rightarrow} R_2$. In particular, 
 $$(t^{\univ}, d^{\univ})=(\tr \rho_1^{\univ}, \det \rho_1^{\univ})=(\tr \rho_2^{\univ}, \det \rho_2^{\univ}).$$ A key ingredient, in the most difficult 
 $p=2$ case, is the computation of $D^{\ps}(k[\varepsilon])$ done in Proposition \ref{tangent_space_seq}, where we follow very closely an argument of Bella\"iche \cite{bel}, and the description by Chenevier of $R_1$ and $R_2$ in \cite{che_lieu}. In fact Chenevier  has already shown that the maps are surjective, and 
 become isomorphism after inverting $2$.
 
In \S \ref{versal_ring} we compute the versal deformation ring $R^{\ver}$ of $\chi_1\oplus \chi_2$. We show that 
$$R^{\ver}\cong R^{\ps}[[x,y]]/(xy-c),$$
where $c\in R^{\ps}$ generates the reducibility ideal. This has been observed by Yongquan Hu and Fucheng Tan in \cite{hu_tan}, for $p>2$, 
using results of B\"ockle, \cite{boeckle}, which, as explained above, involves writing down matrices of ``the most general form"
for the topological generators. Our proof works as follows. If $\rho: \gal\rightarrow \GL_2(k)$ is a continuous representation with semi-simplification isomorphic to $\chi_1\oplus \chi_2$ then any 
 lift of $\rho$ to $A\in \mathfrak A$, $\rho_A: \gal\rightarrow \GL_2(A)$ is naturally an $\CH(R^{\ps})$-module. The idea is that if 
 one understands the algebra $\CH(R^{\ps})$ well, one should be able just to write down the ``most general" deformation of $\rho$. 
 In view of structural results on Cayley--Hamilton algebras by Bella\"iche--Chenevier \cite[\S1.4.3]{bel_che}, we expect that this idea
 will be applicable in other contexts. 
 
 If $p=2$ then  using the description of $R^{\ver}$ above, we observe in Remark \ref{remark_BJ} that 
 $R^{\ver}$ has two irreducible components which, via the map induced by taking determinants, correspond to the two irreducible components
of the universal deformation ring of $1$-dimensional representation $\chi_1\chi_2$. This verifies a conjecture of B\"ockle and Juschka \cite{boe_ju}
in this case. 
 
 In \S \ref{BM} we show that the Breuil--M\'ezard conjecture formulated in \cite{bm}, which describes the Hilbert--Samuel multiplicities of 
 potentially semi-stable
 deformation rings, for $\rho_1$ and $\rho_2$ implies the Breuil--M\'ezard conjecture for the residual representation $\chi_1\oplus \chi_2$. 
 If $p>2$ then the Breuil--M\'ezard conjecture in these cases has been proved by Kisin \cite{kisinfm} as a part of his proof of Fontaine--Mazur 
 conjecture. In \cite{duke}, again under assumption $p>2$, we have given a different local proof for the residual representations $\rho_1$ and $\rho_2$.
  Yongquan Hu and Fucheng Tan observed  in \cite{hu_tan} that the Breuil--M\'ezard conjecture for $\rho_1$ and $\rho_2$ implies the result for 
  $\chi_1\oplus \chi_2$, thus obtaining a local proof also in the generic split case. They use results of B\"ockle to describe the versal deformation ring, 
  and this forces them to assume $p>2$. We use our  description of the versal ring, which works for all $p$, and closely follow their argument.
   The upshot is that in the companion paper \cite{p2} we apply the formalism developed in \cite{duke} to prove the Breuil--M\'ezard conjecture for 
 $\rho_1$ and $\rho_2$, when $p=2$, and this paper implies the result in the split non-scalar case. We formulate our results in the language of 
 cycles, as introduced by Emerton--Gee in \cite{emertongee}.

\textit{Acknowledgements.} A large part of the paper was written while visiting Michael Spie{\ss} at the University of Bielefeld. I thank SFB 701 for generous support 
during my visit. 

\section{Notation}

Let $D^{\ps}:\mathfrak A \rightarrow Sets$ be a functor, which maps  $(A,\mm_A)\in \mathfrak A$ to the set of pairs of  functions
 $(t, d): \gal\rightarrow A$, such 
that the following hold: $d:\gal \rightarrow A^{\times}$ is a continuous group homomorphism, congruent to $\chi_1\chi_2$ modulo $\mm_A$, 
$t: \gal\rightarrow A$ is a continuous function with $t(1)=2$, and, which satisfy for all $g, h\in \gal$:
\begin{itemize} 
\item[(i)] $t(g)\equiv \chi_1(g)+\chi_2(g)\pmod{\mm_A}$;
\item[(ii)] $t(gh)=t(hg)$;
\item[(iii)] $d(g) t(g^{-1}h)-t(g)t(h)+ t(gh)=0$.
\end{itemize}
Such a pair $(t, d): \gal\rightarrow A$ corresponds to an $A$-valued $2$-dimensional determinant in the sense of \cite[Def. 1.15]{che_det}, see \cite[Ex. 1.18]{che_det}.  Given such a pair, we let $J$ be the closed two-sided ideal in $A[[\gal]]$ generated by the elements $g^2-t(g)g + d(g)$ for all $g\in \gal$, and 
let 
$$ \CH(A):= A[[\gal]]/J$$
be the corresponding Cayley--Hamilton algebra. 

Let $H$ be the image and $K$ be the kernel of the group homomorphism $\gal\rightarrow k^{\times}\times k^{\times}$, $g\mapsto (\chi_1(g), \chi_2(g))$.
Let $P$ be the maximal pro-$p$ quotient of $K$ and let $G$ be the quotient of $\gal$ fitting into the exact sequence $1\rightarrow P\rightarrow G\rightarrow H\rightarrow 1$. It follows from \cite[Lem. A.1]{cmf} that the map $A[[K]]\rightarrow \CH(A)$ factors through $A[[P]]\rightarrow \CH(A)$. In particular, $t(g)$ and $d(g)$ 
depend only on the image of $g$ in $G$, and this induces an isomorphism of algebras $A[[G]]/J\cong \CH(A)$, where  $J$ is the closed two-sided ideal of $A[[G]]$ generated by the elements $g^2-t(g)g + d(g)$, for all $g\in G$. Since $H$ is a finite group of order prime to $p$ and $P$ is pro-$p$, the surjection $G\twoheadrightarrow H$ has a splitting, so that $G\cong P\rtimes H$. We let 
$$e_{\chi_1}:=\frac{1}{|H|} \sum_{h\in H} [\chi_1](h) h^{-1}, \quad e_{\chi_2}:=\frac{1}{|H|} \sum_{h\in H} [\chi_2](h) h^{-1},$$
where the square brackets denote the Teichm\"uller lifts to $\OO$. We will denote by the same letters the images of these elements in $\CH(A)$.

\section{Cayley--Hamilton algebras}
\begin{lem}\label{compute}  There is an isomorphism of $\gal$-representations:
$$\CH(k)e_{\chi_2} \cong \rho_1, \quad \CH(k)e_{\chi_1}\cong  \rho_2, \quad \CH(k)\cong \rho_1\oplus \rho_2.$$
\end{lem}
\begin{proof}  The $\gal$-cosocle of $\rho_1\oplus \rho_2$ is $\chi_2\oplus \chi_1$. Since these
characters are distinct, $\rho_1\oplus \rho_2$ is a cyclic $k[[\gal]]$-module. Moreover, elements $g^2- (\chi_1(g)+\chi_2(g))g + \chi_1\chi_2(g)$ kill $\rho_1\oplus \rho_2$, and hence we obtain a surjection of $\gal$-representations $\CH(k)\twoheadrightarrow \rho_1\oplus \rho_2$.

Since the order of $H$ is prime to $p$, $\CH(k)$ is semi-simple as an $H$-representation. If $H$ acts on $v\in \CH(k)$ by a character 
$\psi$, then for all $h\in H$, we have 
$$0= (h-\chi_1(h))(h- \chi_2(h))v= (\psi(h)-\chi_1(h))(\psi(h)-\chi_2(h)) v.$$
Let $H_1:=\{h\in H: \psi(h)=\chi_1(h)\}$, $H_2=\{h\in H: \psi(h)=\chi_2(h)\}$. Then $H_1$ and $H_2$ are subgroups of $H$, such that $H_1 \cup H_2=H$. 
This implies, for example by calculating $|H|$ as $|H_1|+ |H_2| - |H_1\cap H_2|$ and $|H_1| |H_2|/|H_1\cap H_2|$, that either $H_1=H$ or  
or $H_2=H$. Thus either $\psi=\chi_1$ or $\psi=\chi_2$.   

Let $I_P$ be the augmentation ideal in $k[[P]]$. Since $P$ is normal, $\CH(k)/I_P \CH(k)\cong k[H]/ \overline{J}$, where 
$\overline{J}$ is the  two-sided ideal generated by all the elements of the form $h^2-(\chi_1(h)+\chi_2(h)) h + \chi_1\chi_2(h)$, for all $h\in H$.
We know that $k[H]/\overline{J}$ admits $(\rho_1\oplus \rho_2)/ I_P( \rho_1\oplus \rho_2)\cong \chi_2\oplus \chi_1$ as a quotient.
Since  $H$ is abelian, $\chi_1$ and $\chi_2$ occur in $k[H]$ with multiplicity one. Moreover, the argument above shows that no other
character can occur in  $k[H]/ \overline{J}$. Hence, we have an isomorphism of $G$-representations 
$\CH(k)/I_P \CH(k)\cong \chi_2\oplus \chi_1$.

Topological Nakayama's lemma implies that $\CH(k)$ is generated as $k[[P]]$-module by the two elements $e_{\chi_1}$ and $e_{\chi_2}$ defined in the previous section.
Let $M_1$ be the $k[[P]]$-submodule of $\CH(k)$ generated by $e_{\chi_2}$, and let $M_2$ be the $k[[P]]$-submodule 
of $\CH(k)$ generated by $e_{\chi_1}$, so that $M_1= \CH(k)e_{\chi_2}$ and $M_2=\CH(k)e_{\chi_1}$.
We claim that $M_1\cong \rho_1$ and $M_2\cong \rho_2$ as $G$-representations. 
The claim implies that the surjection $\CH(k)\twoheadrightarrow \rho_1\oplus\rho_2$ is an isomorphism.

We will show the claim for $M_1$, the proof for $M_2$ is the same.  We know that $I_P M_1/ I_P^2 M_1$ as an $H$-representation is a direct sum 
of copies of $\chi_1$ and $\chi_2$. Since $\Ext^1_{\gal}(\chi_2, \chi_1)$ is one dimensional, $\chi_1$ appears with multiplicity $1$. If $\chi_2$ appears
in $I_P M_1/ I_P^2 M_1$, then $M_1$ would admit a quotient $N$, which is a non-split extension of $\chi_2$ by itself as a $G$-representation. 
If $p\in P$ is such that $p$ does not act trivially on $N$, and $h\in H$ is such that $\chi_1(h)\neq \chi_2(h)$ then the minimal polynomial of $g:=hp$ acting on  $N$ 
is $(x- \chi_2(g))^2$. Since $(g-\chi_1(g))(g-\chi_2(g))$ kills $\CH(k)$, it will also kill $N$. Since $\chi_1(g)\neq \chi_2(g)$, we get a contradiction. Hence, 
$M_1/I^2_P M_1\cong \rho_1$. Since $I_P M_1/I_P^2 M_1$ is a one dimensional $k$-vector space on which $H$ acts by $\chi_1$, and 
$\Ext^1_{\gal}(\chi_1, \chi_2)$ is one dimensional, the same argument shows that if $I^2_P M_1\neq 0$, then $I^2_P M_1/I_P^3 M$ is one dimensional, 
and $H$ acts on it by $\chi_2$. Hence, it is enough to show that $e_{\chi_2} I_P M_1 =0$, since then Nakayama's lemma would imply that $I^2_P M_1=0$. From 
$(gh)^2- (\chi_1(h)+\chi_2(h)) gh + \chi_1\chi_2(h)=0$, we get that the following holds in $\CH(k)$: 
$$gh= (\chi_1(h)+\chi_2(h))- \chi_1\chi_2(h) h^{-1} g^{-1}, \quad \forall h\in H,  \quad \forall g\in P.$$ 
Then 
\begin{equation} 
\begin{split} 
e_{\chi_2} g e_{\chi_2}&= \frac{1}{|H|^2} \sum_{h_1, h_2\in H} \chi_2(h_1^{-1})\chi_2(h_2^{-1}) h_1 g h_2\\
&=  \frac{1}{|H|^2} \sum_{h_1, h_2\in H} \chi_2(h_1^{-1})\chi_2(h_2^{-1}) h_1 (\chi_1(h_2)+\chi_2(h_2))\\ &- 
 \frac{1}{|H|^2} \sum_{h_1, h_2\in H} \chi_2(h_1^{-1})\chi_1(h_2) h_1 h_2^{-1} g^{-1}\\
 &= e_{\chi_2} - \frac{1}{|H|^2} \sum_{h_1, k\in H}  \chi_2(h_1^{-1}) \chi_1(k^{-1} h_1) k g= e_{\chi_2},
\end{split}
\end{equation} 
where we have used the orthogonality of characters. Hence, $e_{\chi_2} (g-1) e_{\chi_2}=0$ in $\CH(k)$ for all $g\in P$, and so $e_{\chi_2} I_P M_1=0$.
\end{proof}

\begin{lem}\label{k_alg} There is an isomorphism of $k$-algebras:
$$\CH(k)^{\op}\cong \End_{\gal}(\rho_1\oplus \rho_2)\cong \begin{pmatrix} k e_{\chi_1}  & k \Phi_{12} \\ k \Phi_{21} & k e_{\chi_2} \end{pmatrix}, $$
where $\Phi_{12}$, $\Phi_{21}$ are elements of $\CH(k)$ such that $\{e_{\chi_1}$, $\Phi_{21}$, $e_{\chi_2}$, $\Phi_{12}\}$ is a basis of $\CH(k)$ as a $k$-vector
space and  the following relations hold:
$$\Phi_{12}= e_{\chi_1} \Phi_{12}= \Phi_{12} e_{\chi_2},\quad \Phi_{21}= e_{\chi_2} \Phi_{21}= \Phi_{21} e_{\chi_1},\quad  \Phi_{12} \Phi_{21}= \Phi_{21} \Phi_{12}=0.$$

Let $c_{12}, c_{21}: \gal\rightarrow k$ be the functions, such that  an element $g\in \gal$ is mapped to $\chi_1(g)e_{\chi_1} + c_{12}(g) \Phi_{12} + c_{21}(g) \Phi_{21}+ \chi_2(g) e_{\chi_2}$
under the natural map $k[[\gal]]\twoheadrightarrow \CH(k)$. Then $c_{12}$, $c_{21}$ are $1$-cocycles, such that the image of $c_{12}$ in $\Ext^1_{\gal}(\chi_2, \chi_1)$, and the image of 
$c_{21}$ in $\Ext^1_{\gal}(\chi_1, \chi_2)$, span the respective vector space. 
\end{lem} 
\begin{proof} The multiplication on the right by elements of $\CH(k)$ induces an injection of $k$-algebras $\CH(k)^{\op}\hookrightarrow \End_{\gal}(\CH(k))$. 
Since $\CH(k)\cong \rho_1\oplus \rho_2$ by Lemma \ref{compute}, we deduce that both $\CH(k)$ and $\End_{\gal}(\CH(k))$ 
are $4$-dimensional $k$-vector spaces. Hence  the injection is an isomorphism.  Since $\CH(k) e_{\chi_1}\cong \rho_2$ and $\CH(k)e_{\chi_2}\cong 
\rho_1$ by Lemma \ref{compute}, and the restrictions of $\rho_1$ and $\rho_2$ to $H$ are isomorphic to $\chi_1\oplus \chi_2$, we deduce that 
 $$e_{\chi_1} \CH(k) e_{\chi_1},\quad  e_{\chi_2} \CH(k) e_{\chi_1},\quad  e_{\chi_2} \CH(k)e_{\chi_2}, \quad e_{\chi_1} \CH(k)e_{\chi_2}$$ 
are all $1$-dimensional $k$-vector spaces. The idempotents $e_{\chi_1}$, $e_{\chi_2}$ are basis vectors of $e_{\chi_1} \CH(k) e_{\chi_1}$ and 
$e_{\chi_2} \CH(k) e_{\chi_2}$,  respectively. We choose a basis element $\Phi_{21}$ of $e_{\chi_2} \CH(k) e_{\chi}$ and a basis element  $\Phi_{12}$ of $e_{\chi_1} \CH(k) e_{\chi_2}$. It is immediate that the claimed relations 
are satisfied. 

Let $\bar{g}:=g + J$ be the image of $g\in \gal$ in $\CH(k)$. Since $\{e_{\chi_1}$, $\Phi_{21}$, $e_{\chi_2}$, $\Phi_{12}\}$ is a basis of $\CH(k)$ as a $k$-vector
space, we may write 
$$ \bar{g}= c_{11}(g) e_{\chi_1} + c_{12}(g) \Phi_{12}(g) + c_{21}(g) \Phi_{21}(g) + c_{22}(g) e_{\chi_2}\text{ with } c_{ij}(g)\in k.$$
 The left action 
of $\gal$ on $\CH(k)e_{\chi_1}$ factors through the action of $\CH(k)$ and hence $g$ acts as $\bar{g}$. The multiplication relations imply that  $\bar{g} e_{\chi_1}= c_{11}(g) e_{\chi_1}+ c_{21}(g) \Phi_{21}$ and $\bar{g} \Phi_{21}= c_{22}(g) \Phi_{21}$. Thus the left action of $\gal$ on $\CH(k)e_{\chi_1}$ with respect to the basis $\{ e_{\chi_1}, \Phi_{21}\}$ is given by 
$ g\mapsto \bigl (\begin{smallmatrix} c_{11}(g) & 0 \\ c_{21}(g) & c_{22}(g)\end{smallmatrix} \bigr)$. Since Lemma \ref{compute} tells us that this representation is isomorphic to $\rho_2$, which is
a non-split extension of distinct characters,  
we deduce that $c_{11}(g)=\chi_1(g)$, $c_{22}(g)=\chi_2(g)$ and $c_{21}$ is a $1$-cocyle, whose image in $\Ext^1_{\gal}(\chi_1, \chi_2)$ corresponds to $\rho_2$. Since $\rho_2$ is 
non-split, the image of $c_{21}$ is non-zero. Since by assumption $\chi_1\chi_2^{-1}\neq \Eins, \omega^{\pm 1}$, $\Ext^1_{\gal}(\chi_1, \chi_2)$ is a $1$-dimensional $k$-vector space   and hence the image of $c_{21}$ is a basis vector. The same argument with $\CH(k) e_{\chi_2}$ instead of $\CH(k)e_{\chi_1}$ proves the analogous assertion about $c_{12}$.
\end{proof}

\begin{lem}\label{free} Let $(t, d)\in D^{\ps}(A)$ and let $\CH(A)$ be the corresponding Cayley--Hamilton algebra. Then $e_{\chi_1} \CH(A) e_{\chi_1}$ and 
$e_{\chi_2} \CH(A) e_{\chi_2}$ are free $A$-modules of rank $1$ with generators $e_{\chi_1}$, $e_{\chi_2}$ respectively. 
\end{lem}
\begin{proof} We will show the statement for $\chi_1$, the proof for $\chi_2$ is the same. We have 
$$e_{\chi_1} \CH(A) e_{\chi_1}/\mm_A e_{\chi_1} \CH(A) e_{\chi_1}= e_{\chi_1} \CH(k)e_{\chi_1},$$ 
which is a $1$-dimensional vector space spanned by $e_{\chi_1}$ by Lemma \ref{k_alg}. Nakayama's lemma implies that 
$e_{\chi_1} \CH(A) e_{\chi_1}$ is a cyclic $A$-module with generator $e_{\chi_1}$. It is enough to construct a surjection of $A$-modules onto $A$. 

Since $t$ is continuous,  we may extend it to a map of $A$-modules, $t: A[[ \gal]]\rightarrow A$. If $h, g\in \gal$ then 
$t( h (g^2- t(g)g +d(g)))= t(hg^2) - t(g) t(gh) +d(g) t(h)=0$, using the property (iii) above. Hence, the map factors through $t: \CH(A)\rightarrow A$. Since $t(e_{\chi_1}) \pmod{\mm_A}= \tr \rho_2(e_{\chi_1})=1$, $t(e_{\chi_1})$ is a unit in $A$ and 
the map is surjective. Hence, $t$ induces a surjection of $A$-modules 
$e_{\chi_1} \CH(A) e_{\chi_1}$ onto $A$.
\end{proof}

\begin{prop}\label{tangent_space_seq} Let $k[\varepsilon]$ be the dual numbers over $k$. There is an exact sequence of $k$-vector spaces: 
\begin{equation}\label{tangent}
0\rightarrow \Ext^1_{\gal}(\chi_1, \chi_1)\oplus \Ext^1_{\gal}(\chi_2, \chi_2)\rightarrow D^{\ps}(k[\varepsilon])\rightarrow
 \Ext^1_{\gal}(\chi_1, \chi_2)\otimes \Ext^1_{\gal}(\chi_2, \chi_1).
 \end{equation}
 \end{prop}
 \begin{proof} The proof is essentially the same as \cite[Thm. 2]{bel}. We supply the details, since in \cite{bel} the  reference \cite{bel_che} is used, where some assumptions on $p$  are made.
 
 We will define the first non-trivial arrow in \eqref{tangent}. Let $D_{\chi_1}$ and $D_{\chi_2}$ be the deformation functors of $\chi_1$ and $\chi_2$, respectively. Sending 
 $\{\tilde{\chi}_1, \tilde{\chi}_2\}\mapsto (\tilde{\chi}_1+\tilde{\chi}_2, \tilde{\chi}_1\tilde{\chi}_2)$ induces a map
 $D_{\chi_1}(k[\varepsilon])\times D_{\chi_2}(k[\varepsilon])\rightarrow D^{\ps}(k[\varepsilon])$. We may recover $\tilde{\chi}_1$ and $\tilde{\chi}_2$ from $t$, as 
 $$ \tilde{\chi}_1(g)= t(e_{\chi_1} g), \quad \tilde{\chi}_2(g)=t(e_{\chi_2} g).$$
 Hence, the map is injective. The first arrow is obtained by identifying $D_{\chi_1}(k[\varepsilon])$ and $D_{\chi_2}(k[\varepsilon])$ with 
 $\Ext^1_{\gal}(\chi_1, \chi_1)$ and $\Ext^1_{\gal}(\chi_2, \chi_2)$ respectively. 
 
 We will define the last arrow in \eqref{tangent}. Let $(t,d)\in D^{\ps}(k[\varepsilon])$ and let $\CH(k[\varepsilon])$ be the corresponding Cayley--Hamilton algebra. Reducing modulo $\varepsilon$ induces an isomorphism  $\CH(k[\varepsilon])/ \varepsilon \CH(k[\varepsilon])\cong \CH(k)$. Let 
 $\tilde{\Phi}_{12}, \tilde{\Phi}_{21}\in \CH(k[\varepsilon])$ be lifts of $\Phi_{12}$ and $\Phi_{21}$ such that $\tilde{\Phi}_{12}= e_{\chi_1}\tilde{\Phi}_{12}= 
\tilde{\Phi}_{12} e_{\chi_2}$,  $\tilde{\Phi}_{21}= e_{\chi_2}\tilde{\Phi}_{21}= 
\tilde{\Phi}_{21} e_{\chi_2}$. Then $\tilde{\Phi}_{12}$ is a generator for $e_{\chi_1} \CH(k[\varepsilon])e_{\chi_2}$ and $\tilde{\Phi}_{21}$ is a generator 
of $e_{\chi_2} \CH(k[\varepsilon]) e_{\chi_1}$ as an $A$-module. Since $e_{\chi_2} \tilde{\Phi}_{21}\tilde{\Phi}_{12} e_{\chi_2}=  \tilde{\Phi}_{21}\tilde{\Phi}_{12}$, 
and $\Phi_{12} \Phi_{21}=0$, there is a unique $\lambda_{t,d}\in k$, such that  $\tilde{\Phi}_{21}\tilde{\Phi}_{12}= \varepsilon \lambda_{t,d} e_{\chi_2}$. Since
$(1+ \varepsilon \mu) \tilde{\Phi}_{21}\tilde{\Phi}_{12}= (1+\varepsilon \mu) \varepsilon \lambda_{t,d} e_{\chi_2}= \varepsilon \lambda_{t,d} e_{\chi_2}$, $\lambda_{t,d}$ does not depend on the choice of the lift. For all $g, h\in \gal$ we have 
$$ e_{\chi_2} g e_{\chi_1} h e_{\chi_2}= \tilde{c}_{21}(g) \tilde{\Phi}_{21} \tilde{c}_{12}(h) \tilde{\Phi}_{21}=  \varepsilon \lambda_{t,d} c_{21}(g) c_{12}(h) e_{\chi_2}$$
and hence 
$$\varepsilon \lambda_{t, d}c_{21}(g) c_{12}(h) e_{\chi_2}= t(e_{\chi_2} g e_{\chi_1} h e_{\chi_2}).$$
 The map $(t, d)\mapsto \lambda_{t,d} c_{21} c_{12}$ defines the last arrow.  
 
 Given  a function $f: \gal\rightarrow k[\varepsilon]$, we define $f_0, f_1: G\rightarrow k$ by $f(g)= f_0(g) +\varepsilon f_1(g)$, so that  $t_0=\chi_1+\chi_2$
 and $d_0=\chi_1\chi_2$. The $k$-vector space structure in $D^{\ps}(k[\varepsilon])$ is given by 
 $\lambda (t, d)+ \mu( t', d')= (t_0 +\varepsilon(\lambda t_1 + \mu t'_1), d_0+ \varepsilon(\lambda d_1+ \mu d'_1))$. Since 
 $t(e_{\chi_2} g e_{\chi_1} h e_{\chi_2})=\varepsilon t_1(e_{\chi_2} g e_{\chi_1} h e_{\chi_2})$, the last arrow in \eqref{tangent} is $k$-linear. 
 
 If $t=\tilde{\chi}_1+\tilde{\chi}_2$ then using orthogonality of characters we get that $t(e_{\chi_2} g e_{\chi_1} h e_{\chi_2})=0$, for all $g, h\in \gal$. Hence \eqref{tangent} is a complex. 
 
 If $(t,d)$ is mapped to zero, then $\tilde{\Phi}_{21} \tilde{\Phi}_{12}=0$. Hence, the $k[\varepsilon]$-module generated by $\tilde{\Phi}_{12}$ is stable 
 under the action of $\CH(k[\varepsilon])$. It follows from Lemma \ref{free} that the quotient $\CH(k[\varepsilon])e_{\chi_2}/ k[\varepsilon] \tilde{\Phi}_{12}$
 is a free $k[\varepsilon]$-module of rank $1$ generated by $e_{\chi_2}$. The group $\gal$ acts on the module by the character $\tilde{\chi}_2: \gal\rightarrow 
 k[\varepsilon]^{\times}$, which is a deformation of $\chi_2$. Since $g^2-t(g)g +d(g)$ will kill the module, we obtain that 
 $\tilde{\chi}_2(g)^2 - t(g) \tilde{\chi}_2(g) + d(g)=0$ in $k[\varepsilon]$. Hence, $t(g)= \tilde{\chi}_2(g) + d(g) \tilde{\chi}_2(g)^{-1}$. Since 
 $d \tilde{\chi}_2^{-1}$ is a deformation of $\chi_1$, we deduce that \eqref{tangent} is exact. 
  \end{proof}

  \begin{lem}\label{tangent_space} If $p=2$ then the last arrow in \eqref{tangent} is zero. Hence, $\dim_k D^{\ps}(k[\varepsilon])=6$.
  \end{lem}
  \begin{proof} The cup product induces a non-degenerate $H$-equivariant alternating pairing 
  \begin{equation}\label{pairing}
  H^1(P, k)\times H^1(P, k)\overset{\cup}{\rightarrow} H^2(P, k).
  \end{equation}
  Since $p=2$, the cyclotomic character modulo $p$ is trivial, and hence $H$ acts trivially on $H^2(P, k)$. Since the order of $H$ 
  is prime to $p$, for any character $\psi$ of $H$, $H^1(G, \psi)$ is the $\psi$-isotypic subspace of $H^1(P, k)$.
   Since $\chi_1\chi_2^{-1}\neq \Eins$ both $H^1(G, \chi_1\chi_2^{-1})$ and $H^1(G, \chi_1^{-1}\chi_2)$ are one dimensional, and
   the pairing is non-degenerate and $H$-equivariant, \eqref{pairing} induces an isomorphism:
  \begin{equation}\label{iso_cup}
  H^1(G, \chi_1\chi_2^{-1})\otimes H^1(G, \chi_1^{-1} \chi_2)\overset{\cong}{\rightarrow} H^2(G, k).
  \end{equation}
 By interpreting the cup product as Yoneda pairing, we deduce from \eqref{iso_cup} that there does not exist 
 a representation $\tau$ of $G$, such that the socle and the cosocle of $\tau$ is isomorphic to $\chi_2$ and 
 and the semi-simplification is isomorphic to $\chi_2\oplus \chi_1\oplus \chi_2$.
 
 If the last arrow in \eqref{tangent} was non-zero, then we could construct such $\tau$ as follows: let 
 $\CH(k[\varepsilon])$ be the Cayley-Hamilton algebra, which corresponds to a pair $(t,d)\in D^{\ps}(k[\varepsilon])$, which
 does not map to zero under the last arrow in \eqref{tangent}. Then $\CH(k[\varepsilon])e_{\chi_2}/\varepsilon \tilde{\Phi}_{12}$
 would be such representation. 
  \end{proof}  
  
Sending a representation $\rho$ to the pair $(\tr \rho, \det \rho)$ induces a natural transformations $D_1\rightarrow D^{\ps}$, $D_2\rightarrow D^{\ps}$, 
and hence homomorphisms of local $\OO$-algebras $t_1: R^{\ps}\rightarrow R_1$, $t_2: R^{\ps}\rightarrow R_2$.
  
 \begin{prop}\label{trace_rep}Sending a representation $\rho$ to the pair $(\tr \rho, \det \rho)$ induces isomorphisms between the local 
 $\OO$-algebras
 $t_1: R^\ps \overset{\cong}{\rightarrow} R_1$, $t_2: R^\ps\overset{\cong}{\rightarrow} R_2.$
\end{prop}
\begin{proof} If $p>2$ this is shown in \cite[Prop.B.17]{cmf}, using \cite[Thm. 2]{bel} and the fact that $R_1$ and $R_2$ are formally smooth as an input. We assume that  $p=2$. Chenevier has shown in \cite[Cor.4.4]{che_lieu} that the maps are surjective and become isomorphisms after inverting $2$. The result follows by combining Chenevier's argument 
with Lemma \ref{tangent_space} as we now explain. The group homomorphisms $\det \rho_1^{\univ}: \gal\rightarrow R_1^{\times}$ factors through the maximal abelian quotient of $\gal$.
Composing it with the Artin map of local class field theory, we obtain a continuous group homomorphism $\mathbb Q_2^{\times}\rightarrow R_1^{\times}$. The restriction 
of this map to the subgroup $\mu:=\{\pm 1\}\subset \mathbb Q_2^{\times}$ induces a homomorphism of $\OO$-algebras 
$\Lambda \rightarrow R_1$, where $\Lambda:=\OO[\mu]$ is the group algebra of $\mu$ over $\OO$. The same argument with $d^{\univ}$ instead of $\det \rho_1^{\univ}$, also makes $R^{\ps}$ into a $\Lambda$-algebra. Moreover, 
it is immediate that $t_1$ is a homomorphism of $\Lambda$-algebras. Chenevier shows that $R_1$ is formally smooth over $\Lambda$ of dimension $5$, 
and that $t_1$ is surjective. He proves this last assertion by checking that the map $D_1(k[\varepsilon])\rightarrow D^{\ps}(k[\varepsilon])$ is injective. It follows from 
Lemma \ref{tangent_space}, that both have the same dimension as $k$-vector space, thus the map is bijective. Since $t_1$ is a map of $\Lambda$-algebras and 
$R_1$ is formally smooth over $\Lambda$, we conclude that $R^{\ps}\cong R_1$. In particular, $R^{\ps}\cong \Lambda[[x_1, \ldots, x_5]]$, where  $\Lambda\cong \OO[[y]]/((1+y)^2-1)$.
\end{proof}

  \begin{lem}\label{teich} The restriction of $(t^{\univ}, d^{\univ})$ to $H$ is equal to $([\chi_1] + [\chi_2], [\chi_1\chi_2])$, where 
  square brackets denote the Teichm\"uller lift to $\OO$.
  \end{lem}
  \begin{proof} Let $D^{\ps}_H$ be the deformation functor parameterizing $2$-dimensional determinants  $(d, t): H\rightarrow A$ lifting 
   $(\chi_1+\chi_2, \chi_1\chi_2)$. We claim that the corresponding deformation ring $R^{\ps}_H$ is equal to $\OO$, and $([\chi_1] + [\chi_2], [\chi_1\chi_2])$
   is the universal $2$-dimensional determinant. The claim implies the assertion of the lemma, as it follows from the proof of Proposition \ref{trace_rep}
   that $R^{\ps}$ is a flat $\OO$-algebra. To prove the claim it is enough to verify that $D^{\ps}_H(k[\varepsilon])=D^{\ps}_H(k)$, since 
   $([\chi_1] + [\chi_2], [\chi_1\chi_2])$ gives a characteristic zero point. If $(t, d)\in D^{\ps}_H(k[\varepsilon])$ let $\CH(k[\varepsilon])$ 
   be the corresponding Cayley--Hamilton algebra. We have already shown in the first part of the proof of Lemma \ref{compute} that 
   $\CH(k[\varepsilon])/\varepsilon \CH(k[\varepsilon])$ is a $2$-dimensional $k$-vector space with basis $\{e_{\chi_1}, e_{\chi_2}\}$.
   Nakayama's lemma and Lemma \ref{free} implies that $\CH(k[\varepsilon])$ is a free $k[\varepsilon]$-module with basis $\{e_{\chi_1}, e_{\chi_2}\}$.
   Since $H$ acts on $e_{\chi_1}$ by the character $\chi_1$ and on $e_{\chi_2}$ by the character $\chi_2$ we deduce that $\CH(k[\varepsilon])\cong \CH(k)\otimes_k k[\varepsilon]$ as $k[\varepsilon][H]$-modules. This implies that 
   $(t,d)=(\chi_1+\chi_2, \chi_1 \chi_2)$.
      \end{proof}

\begin{prop}\label{reducibility}
There is an ideal $\rr$ of $R^{\ps}$ uniquely determined by the following universal property: an ideal $J$ of $R^{\ps}$ contains $\rr$ if and only if 
$t^{\univ} \pmod {J} = \psi_1+ \psi_2$,  $d^{\univ} \pmod{J}= \psi_1\psi_2$, where $\psi_1, \psi_2: \gal \rightarrow R^{\ps}/J$, are deformations of $\chi_1$ and $\chi_2$, respectively, to $R^{\ps}/J$.  The ring $R^{\ps}/\rr$ is reduced and $\OO$-torsion free. The ideal $\rr$ is principal, generated by a regular element.
\end{prop}
\begin{proof} If $p>2$ this is proved in \cite[\S B.2]{cmf}. The proof is essentially the same for $p=2$. If  $\chi: \gal\rightarrow k^{\times}$ is a continuous character then 
its deformation problem $D_{\chi}$ is pro-represented by $R_{\chi}$, which isomorphic to the completed group algebra over $\OO$ of the pro-$2$ completion of
$\mathbb{Q}_2^{\times}$, so that $R_{\chi}\cong \OO[[x,y, z]]/((1+y)^2-1)$. Since $\chi_1\neq \chi_2$, using orthogonality of characters, one shows that for each $A\in \mathfrak A$ the map $D_{\chi_1\chi_2}(A)\times D_{\chi}(A) \rightarrow D^{\ps}(A)$, $(d, \psi_1)\mapsto (d\psi_1^{-1}+\psi_1, d)$ is injective. We thus obtain a surjection 
of $\Lambda$-algebras $R^{\ps}\twoheadrightarrow R_{\chi_1\chi_2}\wtimes_{\OO} R_{\chi_1}$. The ideal $\rr$ is precisely the kernel of this map.
Since the rings are isomorphic to $\Lambda[[x_1, \ldots x_5]]$ 
and $\Lambda[[x_1, z_1, x_2, y_2, z_2]]/((1+y_2)^2-1)$, respectively,  this allows us to conclude.
\end{proof}
We will refer to $\rr$ as the reducibility ideal, and to $V(\rr)$ as the reducibility locus.

  Let $(t^{\univ}, d^{\univ})\in D^{\ps}(R^{\ps})$ be the universal object. Let $J$ be the closed two-sided ideal in $R^{\ps}[[\gal]]$ generated by all the elements
of the form $g^2-t^{\univ}(g)g + d^{\univ}(g)$, for all $g\in \gal$, and let 
$$\CH(R^{\ps}):= R^{\ps}[[\gal]]/J.$$

\begin{prop}\label{identify} The isomorphisms $t_1: R^{\ps}\cong R_1$, $t_2: R^{\ps}\cong R_2$ induce isomorphisms of left $R^{\ps}[[\gal]]$-modules:
$$\CH(R^{\ps})e_{\chi_2}\cong \rho^{\univ}_1, \quad \CH(R^{\ps})e_{\chi_1}\cong \rho^{\univ}_2, \quad \CH(R^{\ps})\cong \rho_1^{\univ}\oplus \rho_2^{\univ}.$$
\end{prop}
\begin{proof} Since $\rho_1\oplus \rho_2$ is a cyclic $R^{\ps}[[\gal]]$-module, there is a map of $R^{\ps}[[\gal]]$-modules 
$\phi: R^{\ps}[[\gal]]\rightarrow \rho_1^{\univ}\oplus \rho_2^{\univ}$, such that after composing with  the reduction modulo 
the maximal ideal of $R^{\ps}$, we obtain a surjection $R^{\ps}[[\gal]]\twoheadrightarrow \rho_1\oplus \rho_2$. Let $C$ and 
$K$ be the cokernel and the kernel of $\phi$, respectively. Lemma \ref{compute} implies that $k\wtimes \phi$ is an isomorphism 
between $k\wtimes_{R^{\ps}} \CH(R^{\ps})$ and $k\wtimes_{R^{\ps}} (\rho_1^{\univ}\oplus \rho_2^{\univ})$. This implies 
that $k\wtimes_{R^{\ps}} C=0$ and Nakayma's lemma for pseudo-compact $R^{\ps}$-modules implies that $C=0$. Thus $\phi$ is surjective. 
Since $\rho_1^{\univ}\oplus \rho_2^{\univ}$ is a free $R^{\ps}$-module of rank $4$, we deduce that $k\wtimes_{R^{\ps}} K=0$, and so 
$K=0$. Hence, $\phi$ is an isomorphism. The same argument proves the other assertions.
\end{proof}

\begin{cor}\label{endo_R_gal} There is a natural isomorphism: 
$$\End_{R^{\ps}[[\gal]]}(\rho_1^{\univ}\oplus \rho_2^{\univ})\cong \CH(R^{\ps})^{\op}.$$
\end{cor}
\begin{proof} If $J$ is a two-sided ideal in a ring $A$ then multiplication on the right induces an isomorphism between $\End_A(A/J)$ and 
the algebra opposite to $A/J$. The assertion follows from the last isomorphism in Proposition \ref{identify}.
\end{proof}

\begin{remar} One may check that $g\mapsto d^{\univ}(g) g^{-1}$ induces an involution on $\CH(R^{\ps})$ and hence an 
isomorphism of $R^{\ps}$-algebras between $\CH(R^{\ps})$ and $\CH(R^{\ps})^{\op}$.
\end{remar}

\begin{prop}\label{relation} The algebra $\CH(R^{\ps})$ is a free $R^{\ps}$-module of rank $4$: 
$$ \CH(R^{\ps})\cong \begin{pmatrix} R^{\ps} e_{\chi_1} & R^{\ps} \tilde{\Phi}_{12} \\ R^{\ps} \tilde{\Phi}_{21} & R^{\ps} e_{\chi_2}\end{pmatrix}.$$
The generators satisfy the following relations 
\begin{equation}\label{relation_0}
 e_{\chi_1}^2=e_{\chi_1}, \quad e_{\chi_2}^2=e_{\chi_2}, \quad e_{\chi_1}e_{\chi_2}=e_{\chi_2}e_{\chi_1}=0,
\end{equation}
\begin{equation}\label{relation_1}
 e_{\chi_1} \tilde{\Phi}_{12}=  \tilde{\Phi}_{12}e_{\chi_2}= \tilde{\Phi}_{12}, \quad 
e_{\chi_2} \tilde{\Phi}_{21}=  \tilde{\Phi}_{21}e_{\chi_1}= \tilde{\Phi}_{21},
\end{equation}
\begin{equation}\label{relation_1_bis}
 e_{\chi_2} \tilde{\Phi}_{12}=  \tilde{\Phi}_{12}e_{\chi_1}= 
e_{\chi_1} \tilde{\Phi}_{21}=  \tilde{\Phi}_{21}e_{\chi_2}=\tilde{\Phi}_{12}^2=\tilde{\Phi}_{21}^2=0,
\end{equation}
\begin{equation}\label{relation_2}
 \tilde{\Phi}_{12} \tilde{\Phi}_{21}= c e_{\chi_1}, \quad  \tilde{\Phi}_{21} \tilde{\Phi}_{12}= c e_{\chi_2}.
 \end{equation}
The element $c$ generates the reducibility ideal in $R^{\ps}$. In particular, $c$ is $R^{\ps}$-regular.
\end{prop}
\begin{proof} The last isomorphism in Proposition \ref{identify} and Proposition \ref{trace_rep} imply that $\CH(R^{\ps})$ is a free $R^{\ps}$-module 
of rank $4$. Nakayama's lemma implies that any four elements of $\CH(R^{\ps})$, which map to a $k$-basis of 
$k\otimes_{R^{\ps}} \CH(R^{\ps})\cong \CH(k)$ is an $R^{\ps}$-basis of $\CH(R^{\ps})$. We lift the $k$-basis of $\CH(k)$, described in Lemma \ref{k_alg}
as follows.  Let $\Psi_1, \Psi_2\in \CH(R^{\ps})$ be any lifts of $\Phi_{12}$ and $\Phi_{21}$, respectively. Let $\tilde{\Phi}_{12}=e_{\chi_1} \Psi_1 e_{\chi_2}$, 
$\tilde{\Phi}_{21}=e_{\chi_2} \Psi_2 e_{\chi_1}$. It follows from relations in Lemma \ref{k_alg} that $\tilde{\Phi}_{12}$ maps to $\Phi_{12}$, $\tilde{\Phi}_{21}$ 
maps to $\Phi_{21}$, hence $\{e_{\chi_1}, \tilde{\Phi}_{12}, \tilde{\Phi}_{21}, e_{\chi_2}\}$ is an $R^{\ps}$-basis of $\CH(R^{\ps})$. The relations 
in \eqref{relation_0}, \eqref{relation_1}, \eqref{relation_1_bis} follow from the fact that $e_{\chi_1}$ and $e_{\chi_2}$ are orthogonal idempotents.
Since $e_{\chi_1}\tilde{\Phi}_{12} \tilde{\Phi}_{21}e_{\chi_1}=\tilde{\Phi}_{12} \tilde{\Phi}_{21}$, and  
$e_{\chi_2}\tilde{\Phi}_{21} \tilde{\Phi}_{12}e_{\chi_2}=\tilde{\Phi}_{21} \tilde{\Phi}_{12}$, we deduce that there are $c_1, c_2\in R^{\ps}$, 
such that 
\begin{equation}\label{relation_3}
 \tilde{\Phi}_{12} \tilde{\Phi}_{21}= c_1 e_{\chi_1}, \quad  \tilde{\Phi}_{21} \tilde{\Phi}_{12}= c_2 e_{\chi_2}.
 \end{equation} 
Now $t^{\univ}: \gal\rightarrow R^{\ps}$ extends to a map of $R^{\ps}$-modules $t^{\univ}: \CH(R^{\ps})\rightarrow R^{\ps}$. 
For each $a, b\in \CH(R^{\ps})$ we have $t^{\univ}(ab)= t^{\univ}(ba)$ since this identity holds for $a, b\in R^{\ps}[\gal]$, 
and the image of $R^{\ps}[\gal]$ in $\CH(R^{\ps})$ is dense, and $t^{\univ}$ is continuous. It follows from Lemma \ref{teich} that  $t^{\univ}(e_{\chi_1})=t^{\univ}(e_{\chi_2})=1$. 
This  together with \eqref{relation_3} implies that $c_1= t^{\univ}( \tilde{\Phi}_{12} \tilde{\Phi}_{21})=t^{\univ}( \tilde{\Phi}_{21} \tilde{\Phi}_{12})= c_2=: c.$

We will show that $c$ generates the reducibility ideal in $R^{\ps}$. It then will follow from the last part of Proposition \ref{reducibility} that $c$ is regular.
Since $R^{\ps}/\rr$ is reduced and $\OO$-torsion free by Proposition \ref{reducibility}, 
if the image of $c$ in $R^{\ps}/\rr$ is non-zero then there is a maximal ideal $\nn$ of $(R^{\ps}/\rr)[1/p]$, which does not contain $c$. Since
$\OO_{\kappa(\nn)}[[\gal]]\cong \OO_{\kappa(\nn)} \wtimes_{R^{\ps}} R^{\ps}[[\gal]]$, 
the images of $\OO_{\kappa(\nn)}[[\gal]]$ and $R^{\ps}[[\gal]]$
in $\End_{\OO_{\kappa(\nn)}}(\CH(R^{\ps})e_{\chi_1}\otimes_{R^{\ps}} \OO_{\kappa(\nn)})$ coincide. In particular it will contain the images of  $\tilde{\Phi}_{12}$, $\tilde{\Phi}_{21}$, 
$e_{\chi_1}$, $e_{\chi_2}$.   The action of these elements on $\CH(R^{\ps})e_{\chi_1}$ with respect to the $R^{\ps}$-basis $e_{\chi_1}$, $\tilde{\Phi}_{21}$ is given by the matrices 
$\bigl(\begin{smallmatrix}  0 & c\\ 0 & 0 \end{smallmatrix} \bigr)$, $\bigl(\begin{smallmatrix}  0 & 0\\ 1 & 0 \end{smallmatrix} \bigr)$, $\bigl(\begin{smallmatrix}  1 & 0\\ 0 & 0 \end{smallmatrix} \bigr)$, 
$\bigl(\begin{smallmatrix}  0 & 0\\ 0 & 1  \end{smallmatrix} \bigr)$, respectively.
Since the image 
of $c$ in $\kappa(\nn)$ is non-zero this implies that the map  $\OO[[\gal]] \otimes_\OO {\kappa(\nn)} \rightarrow \End_{\kappa(\nn)}(\CH(R^{\ps})e_{\chi_1}\otimes_{R^{\ps}} \kappa(\nn))$ is surjective
and hence $\rho_2^{\univ}\otimes_{R^{\ps}} \kappa(\nn)\cong \CH(R^{\ps})e_{\chi_1}\otimes_{R^{\ps}} \kappa(\nn)$  is 
an irreducible representation of $\gal$. Moreover,  its trace is equal to the specialization of $t^{\univ}$ at $\nn$. This leads to a contradiction as $\nn$ contains  the reducibility ideal. 
Hence, $c\in \rr$.  For the other inclusion we observe that \eqref{relation_1}, \eqref{relation_2} imply that 
the $R^{\ps}/c$-submodule   of $\rho_1^{\univ} / (c)$ generated by $\tilde{\Phi}_{21}$ is stable under the action of $\CH(R^{\ps})$. Moreover, it is free 
over $R^{\ps}/c$ of rank $1$. 	It follows from the definition of the reducible locus that the homomorphism $R^{\ps}\twoheadrightarrow R^{\ps}/c$ factors through $R^{\ps}/\rr$.
\end{proof}

\section{The centre} 
In this section we compute the ring $\End^{\cont}_{\OO[[\gal]]}(\rho_1^{\univ}\oplus \rho_2^{\univ})$ and show that its centre is naturally 
isomorphic to $R^{\ps}$. This result is used in \cite{p2} to show that the centre of a certain category of $\GL_2(\mathbb Q_2)$ representations
is naturally isomorphic to a quotient of $R^{\ps}$, corresponding to a fixed determinant,  in the same way as \cite[Prop. B.26]{cmf} is used in 
\cite[Cor.8.11]{cmf}.

\begin{lem}\label{endo_def} Let $F$ be a finite extension of $\Qp$, let $\mathcal G_F$ be its absolute Galois group, and let 
$\rho: \mathcal G_F\rightarrow \GL_n(k)$ be a continuous representation, such that $\End_{\mathcal G_F}(\rho)=k$. 
Let $\rho^{\univ}$ be the universal deformation of $\rho$ and let $R$ be the universal deformation ring. Then for pseudo-compact $R$-modules $\md_1$, $\md_2$, the functor $\md\mapsto \md\wtimes_R \rho^{\univ}$ induces an isomorphism 
$$\Hom^{\cont}_R(\md_1, \md_2)\overset{\cong}{\rightarrow} \Hom^{\cont}_{\OO[[\mathcal G_F]]}(\md_1\wtimes_R \rho^{\univ}, \md_2\wtimes_R \rho^{\univ}).$$
In particular, $\End^{\cont}_{\OO[[\mathcal G_F]]}(\rho^{\univ})\cong \End_{R[[\mathcal G_F]]}(\rho^{\univ})\cong R.$
\end{lem}
\begin{proof} We argue  as in \cite[Lem.11.5, Cor.11.6]{cmf} by induction on $\ell(\md_1)+ \ell(\md_2)$ that for finite length modules $\md_1$, $\md_2$ of $R$
the functor $\md\mapsto \md\otimes_R \rho^{\univ}$ induces an
isomorphism $$\Hom_R(\md_1, \md_2)\overset{\cong}{\rightarrow} \Hom_{\OO[[\mathcal G_F]]}(\md_1\otimes_R \rho^{\univ}, \md_2\otimes_R \rho^{\univ})$$
and an injection $$\Ext^1_R(\md_1, \md_2)\hookrightarrow \Ext^1_{\OO[[\mathcal G_F]]}(\md_1\otimes_R \rho^{\univ}, \md_2\otimes_R \rho^{\univ}).$$ 
This last map is well defined in terms of Yoneda extensions, because  $\rho^{\univ}$ is flat over $R$. The induction step follows by looking at long exact sequences, as in the proof of \cite[Lem. A.1]{ext2}.  

To start the induction we need to check that the statement is true for $\md_1=\md_2=k$. The assertion about homomorphisms in this case, comes from
the assumption that $\End_{\mathcal G_F}(\rho)=k$. Consider an extension of $R$-modules, $0\rightarrow k \rightarrow \md\rightarrow k\rightarrow 0$. 
If the corresponding extension $0\rightarrow \rho\rightarrow \md\otimes_R \rho^{\univ}\rightarrow \rho\rightarrow 0$ is split, then $\md$ is killed by $\varpi$. In which case, 
the map $\md\twoheadrightarrow \md/k \hookrightarrow \md$ makes $\md$ into a free rank $1$ module over the dual numbers $k[\varepsilon]$, and hence induces 
a homomorphism $\phi: R\rightarrow k[\varepsilon]$ in $\mathfrak A$. A standard argument in deformation theory shows that 
mapping $\phi$ to the equivalence class of $k[\varepsilon]\otimes_{R, \phi} \rho^{\univ}$ induces a bijection between $\Hom(R, k[\varepsilon])$
and $\Ext^1_{k[\mathcal G_F]}(\rho, \rho)$.  Thus if $\md \otimes_R \rho^{\univ}\cong \rho\oplus \rho$ as an $\OO[\mathcal G_F]$-module then $\md\cong k\oplus k$ and so the map $\Ext^1_R(k, k)\rightarrow \Ext^1_{\OO[\mathcal G_F]}(\rho, \rho)$
is injective.

The general case follows by writing pseudo-compact modules as a projective limit of modules of finite length, see the proof of \cite[Cor.11.6]{cmf}.
The last assertion of the proposition follows by taking $\md_1=\md_2=R$.
\end{proof}

\begin{remar} There is a gap in the proof of \cite[Lem.B.21]{cmf}. The issue is that the ring denoted by $\End_G(\tilde{\rho}_{ij})$ there 
is equal to $\End^{\cont}_{\OO}(R)\times \End^{\cont}_{\OO}(R)$, which is much bigger than $R\times R$. 
\end{remar} 

\begin{prop}\label{ring} $$\End^{\cont}_{\OO[[\gal]]}(\rho_1^{\univ}\oplus\rho_2^{\univ})\cong \End_{R^{\ps}[[\gal]]}(\rho_1^{\univ}\oplus \rho_2^{\univ})\cong \CH(R^{\ps})^{\op}.$$
\end{prop}
\begin{proof} The second isomorphism is given by Corollary \ref{endo_R_gal}. To establish the first isomorphism, it is enough to show that 
the injection 
\begin{equation}\label{is_iso}
\Hom_{R^{\ps}[[\gal]]}(\rho_i^{\univ}, \rho_j^{\univ})\hookrightarrow \Hom_{\OO[[\gal]]}^{\cont}(\rho_i^{\univ}, \rho_j^{\univ}),
\end{equation}
is an isomorphism for all $i, j\in \{1, 2\}$. If $i=j$ then the assertion follows from Lemma \ref{endo_def}.

 We will prove that \eqref{is_iso} is an isomorphism if $i=1$, $j=2$. For a finitely generated $R^{\ps}[[\gal]]$-module $M$ we let 
$$A(M):=\Hom_{R^{\ps}[[\gal]]}(\rho^{\univ}_1, M), \quad B(M):=\Hom_{\OO[[\gal]]}^{\cont}(\rho^{\univ}_1, M).$$
We have a natural inclusion $A(M)\hookrightarrow B(M)$, which is an isomorphism if $M=\rho_1^{\univ}$ by Lemma \ref{endo_def}.
It follows from the description of $\rho_1^{\univ}$ and $\rho_2^{\univ}$ in Proposition \ref{identify} and Proposition \ref{relation} that 
multiplication by $\tilde{\Phi}_{21}$ on the right induces an injection $\rho_2^{\univ}\hookrightarrow \rho_1^{\univ}$. We denote the 
quotient by $Q$. We apply $A$ and $B$ to the exact sequence to get a commutative diagram: 
\begin{displaymath}
\xymatrix@1{ 0 \ar[r]& A (\rho_2^{\univ})
\ar@{^(->}[d]\ar[r] &  A(\rho_1^{\univ}) \ar[r]\ar[d]^{\cong} & A(Q)\ar@{^(->}[d]\\
            0 \ar[r] & B(\rho_2^{\univ}) \ar[r]& B(\rho_1^{\univ})\ar[r]& B(Q).}
\end{displaymath}
The diagram implies that the first vertical arrow is an isomorphism. The same argument with $\rho_1^{\univ}$ instead of $\rho_2^{\univ}$ 
and $\tilde{\Phi}_{12}$ instead of $\tilde{\Phi}_{21}$ shows that \eqref{is_iso} is an isomorphism for $i=2$, $j=1$.
\end{proof}

\begin{cor} The center of $\End^{\cont}_{\OO[[\gal]]}(\rho_1^{\univ}\oplus\rho_2^{\univ})$ is naturally isomorphic to $R^{\ps}$.
\end{cor}
\begin{proof} Proposition \ref{ring} implies that it is enough to compute the center of $\CH(R^{\ps})^{\op}$. An element 
 $\Upsilon\in \CH(R^{\ps})$ maybe expressed uniquely as $a_{11} e_{\chi_1} + a_{12} \tilde{\Phi}_{12} + a_{21} \tilde{\Phi}_{21} + a_{22} e_{\chi_2}$
with $a_{11}$, $a_{12}$, $a_{21}$, $a_{22}\in R^{\ps}$. If $\Upsilon$ lies in the centre it must commute with $e_{\chi_1}$, $e_{\chi_2}$ and 
$\tilde{\Phi}_{12}$. Using \eqref{relation_1}, \eqref{relation_2} we deduce that $a_{12}=a_{21}=0$ and $ca_{11}= c a_{22}$. It follows from 
\ref{relation} that $c$ is a regular element, thus $a_{11}=a_{22}$, and since $e_{\chi_1}+ e_{\chi_2}$ is the identity on $\CH(R^{\ps})$ we deduce 
that $\Upsilon\in R^{\ps}$. On the other hand $R^{\ps}$ is contained in the center of $\CH(R^{\ps})$ by construction. 
\end{proof}   

\begin{remar} If the determinant is fixed throughout then  one may show that the composition
$\OO[[\gal]]\rightarrow R^{\ps, \psi}[[\gal]]\twoheadrightarrow \CH(R^{\ps, \psi})$ is surjective. This can be used to give another proof 
of the results in this section, in the case when the determinant is fixed.
\end{remar}

\section{Versal deformation ring} \label{versal_ring}
In this section we compute the versal deformation ring of the representation $\rho=\chi_1\oplus \chi_2$. Recall, \cite{mazur}, that 
a lift of $\rho$ to $A\in \mathfrak A$ is a continuous representation $\gal\rightarrow \GL_2(A)$ congruent to 
$\rho$ modulo the maximal ideal of $A$. Two lifts are equivalent if they are conjugate by a matrix lying in the kernel of 
$\GL_2(A)\twoheadrightarrow \GL_2(k)$. Let $D^{\ver}: \mathfrak A \rightarrow Sets$ be the functor which sends $A$ to the set of equivalence
classes of lifts of $\rho$ to $A$. We define 
\begin{equation}\label{versal}
R^{\ver}:=R^{\ps}[[x,y]]/(xy-c),
\end{equation}
where $c\in R^{\ps}$ is defined in \eqref{relation_2}. The matrices $\bigl ( \begin{smallmatrix} 1 & 0 \\ 0 & 0\end{smallmatrix} \bigr )$, $ \bigl ( \begin{smallmatrix} 0& y \\ 0 & 0\end{smallmatrix} \bigr )$, $\bigl ( \begin{smallmatrix} 0& 0 \\ x & 0\end{smallmatrix} \bigr )$, $\bigl ( \begin{smallmatrix} 0 & 0 \\ 0 & 1\end{smallmatrix} \bigr )$ in $\End_{R^{\ver}}(R^{\ver}\oplus R^{\ver})$ satisfy the same relations  as $e_{\chi_1}$, $\tilde{\Phi}_{12}$, $\tilde{\Phi}_{21}$, $e_{\chi_2}$ in $\CH(R^{\ps})$, see \eqref{relation_0}, \eqref{relation_1}, \eqref{relation_1_bis}, \eqref{relation_2}. Thus mapping
$$e_{\chi_1}\mapsto \bigl ( \begin{smallmatrix} 1 & 0 \\ 0 & 0\end{smallmatrix} \bigr ), \quad  \tilde{\Phi}_{12}\mapsto \bigl ( \begin{smallmatrix} 0& y \\ 0 & 0\end{smallmatrix} \bigr ), \quad \tilde{\Phi}_{21}\mapsto \bigl ( \begin{smallmatrix} 0& 0 \\ x & 0\end{smallmatrix} \bigr ), \quad e_{\chi_2}\mapsto \bigl ( \begin{smallmatrix} 0 & 0 \\ 0 & 1\end{smallmatrix} \bigr )$$
induces a homomorphism of $R^{\ps}$-algebras $\CH(R^{\ps})\rightarrow \End_{R^{\ver}}(R^{\ver}\oplus R^{\ver})$. By composing it with a natural 
map $\gal\rightarrow \CH(R^{\ps})$ we obtain a representation
$$\rho^{\ver}: \gal\rightarrow \GL_2(R^{\ver}).$$
It is immediate that $\rho^{\ver}\otimes_{R^{\ver}} k\cong \chi_1\oplus \chi_2$. For $A\in \mathfrak A$ we let 
$$h^{\ver}(A):= \Hom_{\widehat{\mathfrak A}}(R^{\ver}, A):=\varinjlim_n \Hom_{\mathfrak A}(R^{\ver}/ \mm^n, A),$$
where $\mm$ is the maximal ideal of $R^{\ver}$. Mapping $\varphi\in h^{\ver}(A)$ to the equivalence  class of $\rho^{\ver}\otimes_{R^{\ver}, \varphi} A$
induces a natural transformation
\begin{equation}\label{alpha}
\alpha: h^{\ver}\rightarrow D^{\ver}.
\end{equation}
 We define an equivalence relation on $h^{\ver}(A)$, by the rule $\varphi_1\sim \varphi_2$ if $\varphi_1$ and $\varphi_2$ agree on $R^{\ps}$ and 
 there is $\lambda\in 1+\mm_A$, such that $\varphi_1(x)=\lambda \varphi_2(x)$, $\varphi_1(y)=\lambda^{-1} \varphi_2(y)$.
 
 \begin{lem}\label{bijective} For all $A\in \mathfrak A$ the natural transformation $\alpha$ induces a bijection between $h^{\ver}(A)/\sim$ and 
 $D^{\ver}(A)$.
 \end{lem}  
 \begin{proof} We first observe that if $\varphi_1\sim \varphi_2$ then the representations $\rho^{\ver}\otimes_{R^{\ver}, \varphi_1} A$, 
 $\rho^{\ver}\otimes_{R^{\ver}, \varphi_2} A$ are conjugate by a matrix of the form 
 $\bigl (\begin{smallmatrix} \lambda & 0 \\ 0 & \lambda^{-1}\end{smallmatrix}\bigr )$, for some $\lambda\in 1+\mm_A$. Hence, $\alpha(\varphi_1)=
 \alpha(\varphi_2)$, and the map is well defined. 
 
 If $\varphi_1, \varphi_2\in h^{\ver}(A)$ are such that $\alpha(\varphi_1)=\alpha(\varphi_2)$, then there is a matrix $M\in \GL_2(A)$, congruent 
 to the identity modulo $\mm_A$, such that 
 $$\rho^{\ver}\otimes_{R^{\ver}, \varphi_1} A= M ( \rho^{\ver}\otimes_{R^{\ver}, \varphi_2} A) M^{-1}.$$
 Hence the representations have the same trace and determinant, which implies that $\varphi_1$ and $\varphi_2$ agree on  $R^{\ps}$. 
 Moreover, since both representations map $h\in H$ to a matrix $\bigl( \begin{smallmatrix} [\chi_1](h) & 0 \\ 0 & [\chi_2](h)\end{smallmatrix}\bigr)$, 
 $M$ has to commute with the image of $H$. This implies that $M$ is a diagonal matrix, and hence $\varphi_1\sim \varphi_2$. Thus 
 the map is injective. 
 
 Let $\rho_A: \gal\rightarrow \GL_2(A)$ be a lift of $\rho$. Since $(\tr \rho_A, \det \rho_A)\in D^{\ps}(A)$, we obtain a map 
 $\varphi: R^{\ps}\rightarrow A$.  
 This allows us to view $\rho_A$ as an $R^{\ps}[[\gal]]$-module, and by Cayley--Hamilton, as an  $\CH(R^{\ps})$-module. In other  words 
 we obtain a homomorphism of $R^{\ps}$-algebras $\rho_A: \CH(R^{\ps})\rightarrow \End_A(A\oplus A)$.
 We may conjugate $\rho_A$ with $M\in \GL_2(A)$, which is congruent to $1$ modulo $\mm_A$, such that every $h\in H$ is mapped to 
 a matrix $\bigl( \begin{smallmatrix} [\chi_1](h) & 0 \\ 0 & [\chi_2](h)\end{smallmatrix}\bigr)$. Thus 
 $\rho_A(e_{\chi_1})=\bigl ( \begin{smallmatrix} 1 & 0 \\ 0 & 0\end{smallmatrix} \bigr )$ and 
 $\rho_A(e_{\chi_2})=\bigl ( \begin{smallmatrix} 0 & 0 \\ 0 & 1\end{smallmatrix} \bigr )$. It follows from \eqref{relation_1} that 
 there are $a_{12}, a_{21}\in \mm_A$, such that $\rho_A(\tilde{\Phi}_{12})= \bigl ( \begin{smallmatrix} 0 & a_{12} \\ 0 & 0\end{smallmatrix} \bigr )$, 
 $\rho_A(\tilde{\Phi}_{21})= \bigl ( \begin{smallmatrix} 0 & 0 \\ a_{21} & 0\end{smallmatrix} \bigr )$. It follows from \eqref{relation_2} that
 $a_{12} a_{21}= \varphi(c)$. Hence, we may extend $\varphi: R^{\ps}\rightarrow A$ to $R^{\ver}$ by mapping $x\mapsto a_{21}$, $y\mapsto a_{12}$.
 It follows by construction that $\alpha(\varphi)$ is the equivalence class of $\rho_A$. Thus the map is surjective.
 \end{proof}

\begin{prop}\label{versal_def} The functor $h^{\ver}$ is a versal hull of $D^{\ver}$, $R^{\ver}$ is a versal deformation ring of $\chi_1\oplus \chi_2$.
\end{prop}
\begin{proof} According to \cite[\S 2]{schlessinger} we have to show that $\alpha$ induces a bijection 
$h^{\ver}(k[\varepsilon])\overset{\cong}{\rightarrow} D^{\ver}(k[\varepsilon])$, and for every surjection $B\twoheadrightarrow A$ in $\mathfrak A$ the map 
$$ h^{\ver}(B)\rightarrow h^{\ver}(A)\times_{D^{\ver}(A)} D^{\ver}(B)$$
is surjective. Both claims follow immediately from Lemma \ref{bijective}.
\end{proof}

\begin{remar}\label{remark_BJ} If $p=2$ then it follows from the description of $R^{\ps}$ in the proof of Proposition \ref{reducibility} that 
$R^{\ver}\cong \Lambda[[z_1, \ldots, z_5, x, y]]/( z_5^2 + 2 z_5 -xy)$. Thus $R^{\ver}$ has two irreducible components 
corresponding to the irreducible components of $\Lambda$.  The universal  deformation ring $R_{\chi_1\chi_2}$ of $1$-dimensional representation $\chi_1\chi_2$ is isomorphic to $\Lambda[[x_1, x_2]]$. The map $R_{\chi_1\chi_2}\rightarrow R^{\ver}$ induced by taking determinants
is a map of $\Lambda$-algebras, and hence induces a bijection between their irreducible components. This verifies a conjecture of B\"ockle and Juschka in this case, see \cite{boe_ju}.
\end{remar}

\section{Potentially semi-stable deformation rings}

Let $R$ be either $R_1$, $R_2$, or $R^{\ver}$, let $\rho$ be either $\rho_1$, $\rho_2$, or 
$\bigl( \begin{smallmatrix} \chi_1 & 0 \\ 0 & \chi_2 \end{smallmatrix} \bigr)$  and let $\rho^{\univ}$ be $\rho^{\univ}_1$, $\rho^{\univ}_2$ or $\rho^{\ver}$, respectively.
If $\pp$ is  a maximal ideal of $R[1/p]$, then its residue field $\kappa(\pp)$ is  a finite extension of $L$. Let $E$ be a finite extension of $L$, with the ring 
of integers $\OO_E$ and unformizer $\varpi_E$. If $x: R\rightarrow E$ is a map of $\OO$-algebras, then $\rho^{\univ}_x:= \rho^{\univ}\otimes_{R, x}E$ is a continuous representation $\rho^{\univ}_x: \gal\rightarrow \GL_2(E)$.  The image of $\gal$ is contained in $\GL_2(\OO_E)$, and reducing this 
representation modulo $\varpi_E$ we obtain $\rho$. 

We say that $x$ is potentially semi-stable if $\rho^{\univ}_x$ is a potentially semi-stable 
representation. In this case, to $\rho^{\univ}_x$ we can associate a pair of integers $\wt=(a,b)$ with $a\le b$, the Hodge--Tate weights, 
and a Weil--Deligne representation $\WD(\rho^{\univ}_x)$. We fix $\wt=(a, b)$ with $a<b$ and an 
a representation $\tau: I_{\Qp}\rightarrow \GL_2(L)$ of the inertia subgroup with an open kernel.
Kisin has shown in \cite{kisin_pst} that the locus of $x$ such that the Hodge--Tate weights of $\rho^{\univ}_x$ are equal to $\wt$ and 
$\WD(\rho^{\univ}_x)|_{I_{\Qp}}\cong \tau$ is closed in $\mSpec R[1/p]$. We will call such points of the $p$-adic Hodge type $(\wt, \tau)$.
We let $\Spec R(\wt, \tau)$ be the closure of these points in $\Spec R$
equipped with the reduced scheme structure. Thus $R(\wt, \tau)$ is a reduced $\OO$-torsion free quotient of $R$, characterized by the property 
that $x\in \mSpec R[1/p]$ lies in $\mSpec R(\wt, \tau)[1/p]$ if and only if $\rho^{\univ}_x$ is of $p$-adic Hodge type $(\wt, \tau)$.

\begin{remar}\label{variants} There are following variants  of the set up above to which our results proved below apply, but we do not state them explicitly: 
one may consider potentially crystalline instead of potentially semi-stable points. In this case we will denote the corresponding ring 
by $R^{\mathrm{cr}}(\wt, \tau)$. One may fix a continuous character $\psi: \gal\rightarrow \OO^{\times}$, and require that the representations 
have determinant equal to $\psi\varepsilon$, where $\varepsilon$ is the cyclotomic character. In this case, we will denote the rings  by $R^{\psi}$ and $R^{\psi}(\wt, \tau)$. Note that a necessary condition for 
$R^{\psi}(\wt, \tau)$ to be non-zero is that $\psi|_{I_{\Qp}}=\varepsilon^{a+b-1} \det \tau$ and $\psi \varepsilon\equiv \chi_1\chi_2 \pmod{\varpi}$. One could also look at potentially crystalline representations with the fixed determinant. 
\end{remar} 

Let $x: R\rightarrow E$ be a map of $\OO$-algebras. It follows from Proposition \ref{reducibility} that the representation $\rho^{\univ}_x$ is 
reducible if and only if the reducibility ideal is mapped to zero under the composition $R^{\ps}\rightarrow R \overset{x}{\rightarrow} E$. 
This implies that if $\rho^{\univ}_x$ is irreducible, then it remains irreducible after extending scalars. Let us assume that $x$ is potentially semi-stable
of $p$-adic Hodge type $(\wt, \tau)$. 
If $\rho_x^{\univ}$ is reducible then it is an extension $0\rightarrow \delta_1\rightarrow \rho_x^{\univ} \rightarrow \delta_2\rightarrow 0$, where 
$\delta_1, \delta_2: \gal\rightarrow \OO_E^{\times}$ are continuous characters. Since  $\rho_x^{\univ}$ is potentially semi-stable, both 
$\delta_1$ and $\delta_2$ are potentially semi-stable. Moreover, we may assume that the Hodge-Tate weight of 
$\delta_1\delta_2^{-1}$ is at least $1$, this holds automatically if the extension is non-split. % by \cite{}.
Following Hu--Tan \cite{hu_tan} we say that $x$ is of reducibility type $1$ if 
$\delta_1\equiv \chi_1 \pmod{\varpi_E}$ (equivalently $\delta_2\equiv \chi_2 \pmod{\varpi_E}$). We say that $x$ is of reducibility type $2$ if 
$\delta_1\equiv \chi_2 \pmod{\varpi_E}$. We say that $x$ is of reducibility type $\irr$, if $\rho_x^{\univ}$ is irreducible.

Let $\ast$ be one of the indices $1$, $2$ or $\irr$, we define $I^{\ver}_{\ast}$ to be the ideal of $R^{\ver}$  and $I^{\ps}$ to be the ideal of $R^{\ps}$
given by 
$$ I^{\ver}_{\ast}:= R^{\ver}\cap \bigcap_x \mm_x, \quad I^{\ps}_{\ast}:= R^{\ps}\cap I^{\ver}_{\ast},$$
where the intersection is taken over all $x\in \mSpec R^{\ver}(\wt, \tau)[1/p]$ of reducibility type $\ast$.

\begin{lem}\label{hyper} Let $A$ be a local noetherian ring. Let $B=A[[x,y]]/(xy-c)$, with 
$c\in \mm_A$. Then $B$ is $A$-flat and $\dim B=\dim A+1$. If $A$ is reduced then $B$ is reduced. 
\end{lem} 
\begin{proof} Let $z=x+y$ then $B=A[[z]][x]/(x^2-zx +c)$. Thus $B$ is a free $A[[z]]$-module of rank $2$. This implies the 
claims about flatness and dimension.  If $\pp$ is a prime of $A$ then 
  $\qq:= \pp A[[z]$ is a prime of $A[[z]]$. Since $z$ is transcendental over 
$\kappa(\pp)$, $z^2-4c$ cannot be zero in $\kappa(\qq)$. Thus $\kappa(\qq)[x]/(x^2-zx+c)$ is reduced, and the subring $A/\pp[[z]][x]/(x^2-zx+c)$ 
is also reduced.  If  $A$ is reduced then we may embed $A$ into $\prod_{\pp} A/\pp$, where the product is taken over all the minimal primes of $A$. 
Hence, $B$ can be embedded into the product of reduced rings $\prod_{\pp} A/\pp[[z]][x]/(x^2-zx+c)$.
\end{proof}

\begin{lem}\label{control_cycle}  Let $A$ be a local noetherian ring and let $\pp\in \Spec A$ be such that $\dim A/\pp=\dim A$. Let $B=A[[x,y]]/(xy-c)$, with 
$c\in \pp$.  Let $\qq$ be the ideal of $B$ generated by $\pp$ and $x$. Then $\qq$ is a prime ideal with $\dim B= \dim B/\qq=\dim A+1$. Moreover, 
 $e(A/\pp)= e(B/\qq)$, $\ell_{A_{\pp}}(A_{\pp})=\ell_{B_{\qq}}(B_{\qq})$
\end{lem}
\begin{proof} Since $B/\qq\cong (A/\pp)[[y]]$ we get that $\qq$ is a prime ideal of $B$,  $e(A/\pp)=e(B/\qq)$ and $\dim B/\qq=\dim A/\pp +1 = \dim A+1=\dim B$, 
where the last equality follows from Lemma \ref{hyper}. Since $B$ is $A$-flat by Lemma \ref{hyper}, $B_{\qq}$ is $A_{\pp}$-flat. Since $\pp B_{\qq}$ is the maximal ideal of $B_{\qq}$, flatness implies that $\ell(A_{\pp})=\ell(B_{\qq})$. 
\end{proof}

\begin{lem}\label{ps_ver} The map \eqref{versal} induces isomorphisms: 
$$R^{\ps}/I_1^{\ps}[[y]]\cong R^{\ver}/I_1^{\ver}, \quad R^{\ps}/I_2^{\ps}[[x]]\cong R^{\ver}/I_2^{\ver},$$
$$R^{\ps}/I^{\ps}_{\irr}[[x,y]]/(xy-c)\cong R^{\ver}/I_{\irr}^{\ver}.$$
\end{lem}
\begin{proof} Let $\ast$ be one of the following indices: $1$, $2$ or $\irr$. Then \eqref{versal} induces a surjection:
\begin{equation}\label{blast}
 R^{\ps}/I^{\ps}_{\ast}[[x,y]]/(xy-c)\cong R^{\ver}/I^{\ps}_{\ast} R^{\ver}\twoheadrightarrow R^{\ver}/I^{\ver}_{\ast}.
 \end{equation} 
 
 We will deal with the irreducible case first. To ease the notation let $A=R^{\ps}/I^{\ps}_{\irr}$ and let $B=R^{\ver}/I^{\ps}_{\irr} R^{\ver}$. 
Lemma \ref{hyper} implies that $B$ is reduced and 
$A$-flat. Since by construction  a subset of $\Spec A[1/p]$ is dense in $\Spec A$, $A$ is $\OO$-torsion free. Flatness implies that $B$ is $\OO$-torsion free, hence
$\Spec B[1/p]$ is dense in $\Spec B$. Since $B[1/p]$ is Jacobson, $\mSpec B[1/p]$ is dense in $\Spec B$. The reducible locus in $\Spec B$ is 
given by $c=0$, and is isomorphic to $\Spec (A/c)[[x, y]]/(xy)$. Now $\dim A/c < \dim A$, since otherwise points of type $\irr$ would have to be dense in $\Spec A/c$. Hence the reducible locus in $\Spec B$ has codimension $1$. Thus the subset $\Sigma'$ of $\mSpec B[1/p]$,  consisting of those maximal ideals, 
which correspond to absolutely irreducible representations, is dense in $\Spec B$. Since an absolutely irreducible representation is determined up to isomorphism by its
trace, $\Sigma'$ lies in the image of $\Spec R^{\ver}/I^{\ver}_{\irr}\rightarrow \Spec B$. Since this map is a closed immersion, it is a homeomorphism. 
Since $B$ is reduced, we obtain the assertion. 

 If $\ast=1$ then $c$ is contained in every maximal ideal of type $1$, and hence $c\in I^{\ps}_1$. It follows from the construction 
 of the versal representation that any maximal ideal of $R^{\ver}(\wt, \tau)[1/p]$ of type $1$ will contain $y$, and any maximal ideal 
 of $R^{\ver}/(I_1^{\ps}, x)[1/p]$ is of type $1$. Thus  \eqref{blast}  induces  a closed immersion 
 $$\Spec R^{\ver}/I_1^{\ver}\hookrightarrow \Spec R^{\ver}/(I_1^{\ps}, x),$$ and $\mSpec R^{\ver}/(I_1^{\ps}, x)[1/p]$ lies in its image. 
Now $R^{\ver}/(I_1^{\ps}, x)\cong R^{\ps}/I^{\ps}_1[[y]]$, and hence is reduced and $\OO$-torsion free. The same argument as in the irreducible 
case allows to conclude. If $\ast=2$ then the argument is the same interchanging $x$ and $y$.
\end{proof}

\begin{lem}\label{cycle_2} The isomorphisms $t_1: R^{\ps}\overset{\cong}{\rightarrow} R_1$, $t_2: R^{\ps}\overset{\cong}{\rightarrow} R_2$ induces isomorphisms:
$R^{\ps}/I_1^{\ps} \cap I_{\irr}^{\ps}\cong R_1(\wt, \tau)$, $R^{\ps}/I_2^{\ps}\cap I_{\irr}^{\ps}\cong R_2(\wt, \tau)$.
\end{lem}
\begin{proof} Since the rings are $\OO$-torsion free and reduced, it is enough to show that the maps induce a bijection 
on maximal spectra after inverting $p$. We will show the statement for $R_1(\wt, \tau)$, the proof for $R_2(\wt, \tau)$ is 
the same. Since $t_1$ is an isomorphism it induces a bijection between $\mSpec R_1[1/p]$ and $\mSpec R^{\ps}[1/p]$ and hence 
it is enough to show that every $x \in \mSpec R_1(\wt, \tau)[1/p]$ is mapped to 
$V(I_1^{\ps}\cap I^{\ps}_{\irr})$ and every $y\in  V(I_1^{\ps}\cap I^{\ps}_{\irr}) \cap \mSpec R^{\ps}[1/p]$ has a preimage in 
$\mSpec R_1(\wt, \tau)[1/p]$.

Let $E$ be a finite extension of $L$ with the ring of integers $\OO_E$ and let $x: R_1(\wt, \tau)\rightarrow E$ be an $E$-valued point 
of $\Spec R_1(\wt, \tau)$. Let $\rho_x:= \rho^{\univ}_1\otimes_{R_1, x} E$. The image of $R_1$ under $x$ is contained in $\OO_E$, and 
we let $\rho_x^0:=\rho^{\univ}_1\otimes_{R_1, x} \OO_E$. Then $\rho_x^0$ is a $\gal$-invariant $\OO_E$-lattice in $\rho_x$, and its 
reduction modulo the uniformizer $\varpi_E$, is isomorphic to $\rho_1$. 

If $\rho_x$ is reducible then it is an extension $0\rightarrow \delta_1\rightarrow \rho_x\rightarrow \delta_2\rightarrow 0$, where $\delta_1, 
\delta_2:\gal\rightarrow E^{\times}$ are continuous characters. This extension is non-split, as the reduction of $\rho^0_x$ modulo $\varpi_E$
is a non-split extension of distinct characters. Moreover, $\delta_1$ is congruent to $\chi_1$  and $\delta_2$ is congruent to $\chi_2$ modulo 
$\varpi_E$. Since $x\in \Spec R_1(\wt, \tau)$, $\rho_x$ is potentially semi-stable, hence both $\delta_1$ and $\delta_2$ are potentially semi-stable, 
and the Hodge-Tate weight of $\delta_1$ is greater than the Hodge-Tate weight of $\delta_2$.

 By conjugating 
$\rho^0_x$ with $\bigl (\begin{smallmatrix} \varpi^n & 0 \\ 0 & 1\end{smallmatrix}\bigr)$, for a suitable  $n\in \ZZ$, we will 
obtain a new $\gal$-invariant $\OO_E$-lattice in $\rho_x$, such that is reduction modulo $\varpi_E$ is congruent to $\chi_1\oplus \chi_2$.
This gives an $\OO_E$-valued point in $\Spec R^{\ver}(\wt, \tau)$, which has the same trace as $\rho_x$. Hence, the map 
$\Spec R_1(\wt, \tau)\rightarrow \Spec R^{\ps}$ maps $x$ into $V(I^{\ps}_1)$.  

Conversely, let $y$ be an $E$-valued point of $R^{\ps}/I^{\ps}_1$, 
then the determinant corresponding to $y$ is a pair $(\delta_1+\delta_2, \delta_1\delta_2)$, such that there is $z: R^{\ver}(\wt, \tau)\rightarrow E$ fitting 
into the exact sequence $0\rightarrow \delta_1\rightarrow \rho^{\ver}\otimes_{R^{\ver}, z} E\rightarrow \delta_2\rightarrow 0$, such that 
$\delta_1\equiv \chi_1\pmod{\varpi_E}$, $\delta_2\equiv \chi_2 \pmod{\varpi_E}$ and the Hodge--Tate weight of $\delta_1\delta_2^{-1}$ is at least $1$. Since $\Ext^1_{\gal}(\delta_2, \delta_1)$ is non-zero, 
there is a non-split extension $0\rightarrow \delta_1\rightarrow \tilde{\rho}\rightarrow \delta_2\rightarrow 0$. Since the Hodge--Tate weight of 
$\delta_1\delta_2^{-1}$ is at least $1$, the representation $\tilde{\rho}$ is potentially semi--stable of $p$-adic Hodge type $(\wt, \tau)$. Since 
$\chi_1\chi_2^{-1}\neq \Eins, \omega^{\pm 1}$, $\delta_1\delta_2^{-1}\neq \Eins, \varepsilon^{\pm 1}$, the extension is in fact potentially crystalline.  

We may choose a $\gal$-invariant 
$\OO_E$-lattice $\rho^0$ in $\rho$, such that $\rho^0\otimes_{\OO_E} k$ is a non-split extension of $\chi_1$ by $\chi_2$. Since 
$\Ext^1_{\gal}(\chi_2, \chi_1)$ is one dimensional, this representation is isomorphic to $\rho_1$, and thus $\tilde{\rho}$ gives us an $E$-valued 
point of $\Spec R_1(\wt, \tau)$. Hence, $y$ lies in the image of $\Spec R_1(\wt, \tau)\rightarrow \Spec R^{\ps}$.

In the irreducible case the argument is easier. If $\rho_x$ is irreducible then after extending scalars to $E':= E[ \sqrt{\varpi_E}]$, we will be able  
to find a $\gal$-invariant $\OO_{E'}$ lattice in $\rho_x\otimes_E E'$ with reduction modulo $\varpi_{E'}$ isomorphic to $\chi_1 \oplus \chi_2$
by arguing in the same way as in the reducible case.   
This gives us an $E'$-valued point in $\Spec R^{\ver}(\wt, \tau)$. As remarked after Remark \ref{variants} $\rho_x$ is absolutely irreducible. Since the representation obtained by extending scalars is irreducible and has the same trace as $\rho_x$, we deduce that the 
image of $x$ in $\mSpec R^{\ps}[1/p]$ lies in $V(I^{\ps}_{\irr})$. 

Conversely, let $y$ be an $E$-valued point of $\Spec R^{\ps}/I^{\ps}_\irr$. Let $x$ be the image of $y$ in $\mSpec R_1[1/p]$ under the map induced by the isomorphism 
$t_1: R^{\ps}\overset{\cong}{\rightarrow} R_1$. It follows from the definition of $I^{\ps}_{\irr}$ that there is a finite extension $E'$ of $E$ and an $E'$-valued point $
z$ of $R^{\ver}(\wt, \tau)$, such that $\rho^{\ver}_z$ is irreducible with trace equal to $t^{\univ}_y$. In particular, $\rho^{\ver}_z$ and $\rho_x$ have the same trace. 
As remarked after Remark \ref{variants} $\rho^{\ver}_z$ is absolutely irreducible, thus the equality of traces implies that 
$\rho^{\ver}_z$ is isomorphic to $\rho_x \otimes_E E'$ as $\gal$-representations. Hence, $\rho_x$ is potentially semi-stable of $p$-adic Hodge type equal to 
$(\wt, \tau)$. This implies that $x$ lies in $\mSpec R_1(\wt, \tau)[1/p]$.
%
  %One may use  \cite[VII.4.5]{colmez} and the proof of  \cite[Lem. 9.5]{PCD2} to choose 
%$\OO_E$-latices with the prescribed reduction modulo $\varpi_E$.
\end{proof}
Recall,  \cite[\S V.A]{mult},  that the group of $d$-dimensional cycles $\mathcal Z_d(A)$ of  a noetherian ring $A$  is a free abelian group generated by $\pp\in \Spec A$ with $\dim A/\pp = d$.

\begin{lem}\label{cycle_1} Let $d=\dim R^{\ver}(\wt, \tau)$ then there is an equality of cycles in $\mathcal Z_d(R^{\ver})$:
$$ z_d( R^{\ver}(\wt, \tau))= z_d( R^{\ver}/I^{\ver}_1\oplus R^{\ver}/I^{\ver}_{\irr} \oplus R^{\ver}/I^{\ver}_2).$$
\end{lem}
\begin{proof} Since $R^{\ver}(\wt, \tau)= R^{\ver}/(I_1^{\ver}\cap I_{\irr}^{\ver} \cap I_{2}^{\ver})$, it is enough to 
show that $\qq\in \Spec R^{\ver}$ with $\dim R^{\ver}/\qq=d$ can lie in the support of at most one 
of the modules $R^{\ver}/I^{\ver}_1$, $R^{\ver}/I^{\ver}_{\irr}$, $R^{\ver}/I^{\ver}_2$. If $\qq$ lies in the support 
of $R^{\ver}/I_{\ast}^{\ver}$, then for dimension reasons it has to be a minimal prime in $V(I^{\ver}_{\ast})$, and thus 
points of type $\ast$ are dense in $V(\qq)$.  If $\ast$ is $1$ or $2$ then $c\in\qq$ and hence $V(\qq)$ does not contain
points of type $\irr$. If $\qq$ lies in the support of both $R^{\ver}/I_1^{\ver}$ and $R^{\ver}/I_2^{\ver}$ then it will lie in 
$V((x, y))$, which has codimension $1$, as follows from Lemma \ref{ps_ver}.
\end{proof}
 
\section{The Breuil--M\'ezard conjecture}\label{BM}
Recall that the reducible locus in $R_1$, $R_2$, $R^{\ps}$ and $R^{\ver}$ is defined by the equation $c=0$. 
The isomorphism $t_1: R^{\ps}\overset{\cong}{\rightarrow} R_1$,   $t_2: R^{\ps}\overset{\cong}{\rightarrow} R_2$, \eqref{versal} 
induce isomorphisms: $$t_1: R^{\ps}/c\overset{\cong}{\rightarrow} R_1/c,\quad  t_2:  R^{\ps}/c\overset{\cong}{\rightarrow} R_2/c, 
\quad R^{\ps}/c[[x, y]]/(xy)\overset{\cong}{\rightarrow} R^{\ver}/c.$$ For  $\pp_1\in \Spec R_1/c$ let $\qq_1$ be the ideal
of $R^{\ver}/c$ defined by $\qq_1:=(t_1^{-1}(\pp_1), x)$. Then $R^{\ver}/\qq_1\cong R_1/\pp[[y]]$, and so 
$\qq_1\in \Spec R^{\ver}/c$, $\dim R^{\ver}/\qq_1= \dim R_1/\pp_1 +1$, $e(R^{\ver}/\qq_1)=e(R_1/\pp_1)$. Similarly,
for $\pp_2\in \Spec R_2/c$ we let  $\qq_2:=(t^{-1}_2(\pp_2), y)$, then $R^{\ver}/\qq_2\cong R_2/\pp[[x]]$, $\dim R^{\ver}/\qq_2= \dim R_2/\pp_2 +1$, $e(R^{\ver}/\qq_2)=e(R_2/\pp_2)$. Hence, for all $0\le i\le \dim R^{\ps}/(c)$ the map $\pp_1\mapsto \qq_1$, $\pp_2\mapsto \qq_2$ induces an injection
\begin{equation}\label{cycle_map}
\alpha: \mathcal Z_i(R_1/(c))\oplus \mathcal Z_i(R_2/(c))\hookrightarrow \mathcal Z_{i+1}(R^{\ver}/(c)).
\end{equation}
Moreover, this map preserves Hilbert--Samuel multiplicities. 

\begin{thm}\label{bm_cycle} Assume that $\Spec R_1(\wt, \tau)/\varpi$ (equivalently, $\Spec R_2(\wt, \tau)/\varpi$) is contained in the reducible locus. Let $d$ be the 
dimension of $R^{\ver}(\wt, \tau)$ then there is an equality of $(d-1)$-dimensional cycles: 
$$ z_{d-1}(R^{\ver}(\wt, \tau)/\varpi)= \alpha( z_{d-2}(R_1(\wt, \tau)/\varpi)+ z_{d-2}(R_2(\wt, \tau)/\varpi)). $$
\end{thm}
\begin{proof} If $M_1$ and $M_2$ are finally generated $d$-dimensional modules over a noetherian ring $R$, such that 
$z_d(M_1)=z_d(M_2)$, and $x\in R$ is both $M_1$- and $M_2$-regular, then $z_{d-1}(M_1/x)= z_{d-1}(M_2/x)$, \cite[2.2.13]{emertongee}.  This fact and Lemma
\ref{cycle_1} imply that 
\begin{equation}
\begin{split}
 z_{d-1}(R^{\ver}(\wt, \tau)/\varpi)= &z_{d-1}(R^{\ver}/(I_1^{\ver}, \varpi))+z_{d-1}(R^{\ver}/(I_{\irr}^{\ver}, \varpi))\\&+z_{d-1}(R^{\ver}/(I_2^{\ver}, \varpi)). 
 \end{split}
 \end{equation}
 Simirlarly from Lemma \ref{cycle_2} one obtains
 \begin{equation}
 z_{d-2}(R_1(\wt, \tau)/\varpi)=z_{d-2}(R_1/(t_1(I_1^{\ps}), \varpi))+z_{d-2}(R_1/(t_1(I_{\irr}^{\ps}), \varpi)).
 \end{equation}
\begin{equation}
 z_{d-2}(R_2(\wt, \tau)/\varpi)=z_{d-2}(R_2/(t_2(I_2^{\ps}), \varpi))+z_{d-2}(R_2/(t_2(I_{\irr}^{\ps}), \varpi)).
 \end{equation}
It is immediate from Lemma \ref{ps_ver} and the definition of $\alpha$ that 
$$ \alpha(z_{d-2}(R_1/(t_1(I_1^{\ps}), \varpi)))= z_{d-1}(R^{\ver}/(I_1^{\ver}, \varpi)),$$
$$\alpha(z_{d-2}(R_2/(t_1(I_1^{\ps}), \varpi)))= z_{d-1}(R^{\ver}/(I_2^{\ver}, \varpi)).$$
The assumption that the special fibre of the potentially semi-stable ring is contained in the reducible locus implies 
that $c$ is nilpotent in $R^{\ps}/(I^{\ps}_{\irr}, \varpi)$. It follows from Lemmas \ref{ps_ver}, \ref{control_cycle} that 
$$z_{d-1}(R^{\ver}/(I^{\ver}_{\irr}, \varpi))=\alpha( z_{d-2}(R_1/(t_1(I_{\irr}^{\ps}), \varpi))+ z_{d-2}(R_2/(t_2(I_{\irr}^{\ps}), \varpi))).$$
\end{proof}

\begin{cor} Under the assumption of Theorem \ref{bm_cycle}, we have an equality of the Hilbert--Samuel multiplicities:
$$e(R^{\ver}(\wt, \tau)/\varpi)=e(R_1(\wt, \tau)/\varpi)+ e(R_2(\wt, \tau)/\varpi).$$
\end{cor}
\begin{proof} This follows from the fact that \eqref{cycle_map} preserves Hilbert--Samuel multiplicities and Theorem \ref{bm_cycle}. 
\end{proof}

In \cite{henniart}, Henniart has shown the existence  of a smooth irreducible  representation $\sigma(\tau)$ (resp. $\sigma^{\mathrm{cr}}(\tau)$) of $K:=\GL_2(\Zp)$ on an $L$-vector space, such that if $\pi$ is a smooth absolutely irreducible infinite dimensional  representation of $G:=\GL_2(\Qp)$  and  $\LLL(\pi)$ is the Weil-Deligne representation attached to $\pi$ by the classical local Langlands  correspondence then $\Hom_K(\sigma(\tau), \pi)\neq 0$ (resp. $\Hom_K(\sigma^{\mathrm{cr}}(\tau), \pi)\neq 0$) if and only if $\LLL(\pi)|_{I_{\Qp}}\cong \tau$ (resp. $\LLL(\pi)|_{I_{\Qp}}\cong \tau$ and the monodromy operator $N=0$). We have $\sigma(\tau)\cong \sigma^{\mathrm{cr}}(\tau)$ in all cases, except if $\tau\cong \chi\oplus \chi$, then $\sigma(\tau)\cong \tilde{\st}\otimes \chi\circ \det$ and $\sigma^{\mathrm{cr}}(\tau)\cong \chi\circ\det$, where $\tilde{\st}$ is the Steinberg representation of $\GL_2(\Fp)$, and we view $\chi$ as a character of $\Zp^{\times}$ via the local class field theory.

We let $\sigma(\mathbf w, \tau):=\sigma(\tau)\otimes \Sym^{b-a-1} L^2\otimes \det^a$. Then $\sigma(\mathbf w, \tau)$ is a finite dimensional $L$-vector space. Since $K$ is compact 
and the action of $K$ on $\sigma(\mathbf w, \tau)$ is continuous, there is a $K$-invariant $\OO$-lattice $\Theta$ in $\sigma(\mathbf w, \tau)$. Then $\Theta/(\varpi)$ is a smooth 
finite length $k$-representation of $K$, and we let $\overline{\sigma(\mathbf w, \tau)}$ be its semi-simplification. One may show that $\overline{\sigma(\mathbf w, \tau)}$ does not depend on the choice of a lattice.  For each smooth irreducible  $k$-representation $\sigma$ of $K$ we let $m_{\sigma}(\mathbf w, \tau)$ be the multiplicity with which $\sigma$ occurs in $\overline{\sigma(\mathbf w, \tau)}$. We let $\sigma^{\mathrm{cr}}(\mathbf w, \tau):=\sigma^{\mathrm{cr}}(\tau)\otimes  \Sym^{b-a-1} L^2\otimes \det^a$ and let $m^{\mathrm{cr}}_{\sigma}(\mathbf w, \tau)$ be the multiplicity of $\sigma$ in $\overline{\sigma^{\mathrm{cr}}(\mathbf w, \tau)}$.

\begin{thm}\label{split_bm} If the determinant is fixed let $d=3$, otherwise let $d=4$. Assume that there are finite sets  
$\{\mathcal C_{1, \sigma}\}_{\sigma} \subset \mathcal Z_{d-2}(R_1/\varpi)$, 
$\{\mathcal C_{2, \sigma}\}_{\sigma} \subset \mathcal Z_{d-2}(R_2/\varpi)$ such that for all $p$-adic Hodge types $(\wt, \tau)$ we have equalities 
$$ z_{d-2}( R_1(\wt, \tau)/\varpi)=\sum_{\sigma} m_{\sigma}(\wt, \tau) \mathcal C_{1, \sigma}, \quad  z_{d-2}( R_2(\wt, \tau)/\varpi)=\sum_{\sigma} m_{\sigma}(\wt, \tau)
 \mathcal C_{2, \sigma}.$$
 $$ z_{d-2}( R_1^{\mathrm{cr}}(\wt, \tau)/\varpi)=\sum_{\sigma} m_{\sigma}^{\mathrm{cr}}(\wt, \tau) \mathcal C_{1, \sigma}, \quad  z_{d-2}( R_2^{\mathrm{cr}}(\wt, \tau)/\varpi)=\sum_{\sigma} m_{\sigma}^{\mathrm{cr}}(\wt, \tau)
 \mathcal C_{2, \sigma}.$$

Then $\Spec R_1(\wt, \tau)/\varpi$, $\Spec R_2(\wt, \tau)/\varpi$ are  contained in the reducible locus of $\Spec R_1$ and  $\Spec R_2$ respectively  and for all $p$-adic Hodge types $(\wt, \tau)$ we have 
$$ z_{d-1}(R^{\ver}(\wt, \tau)/\varpi))=\sum_{\sigma}( m_{\sigma}(\wt, \tau)\alpha( \mathcal C_{1, \sigma}) +m_{2, \sigma}(\wt, \tau)
\alpha( \mathcal C_{2, \sigma})),$$
$$ z_{d-1}(R^{\ver, \mathrm{cr}}(\wt, \tau)/\varpi))=\sum_{\sigma}( m^{\mathrm{cr}}_{\sigma}(\wt, \tau)\alpha( \mathcal C_{1, \sigma}) +m^{\mathrm{cr}}_{\sigma}(\wt, \tau)
\alpha( \mathcal C_{2, \sigma})).$$
In particular, 
$$e(R^{\ver}(\wt, \tau)/\varpi)=\sum_{\sigma}( m_{\sigma}(\wt, \tau) e(\mathcal C_{1, \sigma}) +m_{\sigma}(\wt, \tau)
e( \mathcal C_{2, \sigma})),$$
$$e(R^{\ver, \mathrm{cr}}(\wt, \tau)/\varpi)=\sum_{\sigma}( m^{\mathrm{cr}}_{\sigma}(\wt, \tau) e(\mathcal C_{1, \sigma}) +m^{\mathrm{cr}}_{\sigma}(\wt, \tau)
e( \mathcal C_{2, \sigma})).$$

\end{thm}
\begin{proof} Each $\sigma$ is isomorphic to a representation of the form $\Sym^r k^2 \otimes\det^s$ with $0\le r \le p-1$, and $0\le s\le p-2$. The pair $(r, s)$ is 
uniquely determined by $\sigma$. Let $\wt(\sigma):=(s, s+r+1)$. Then $\sigma^{\mathrm{cr}}(\wt(\sigma), \Eins\oplus \Eins)=\Sym^r L^2 \otimes \det^s$, and
 $\overline{\sigma^{\mathrm{cr}}(\wt(\sigma), \Eins\oplus \Eins)}\cong \sigma$. Thus 
$m_{\sigma}^{\mathrm{cr}}(\wt(\sigma), \Eins\oplus \Eins)=1$, and $m_{\sigma'}^{\mathrm{cr}}(\wt(\sigma), \Eins \oplus \Eins)=0$ for all $\sigma'\not\cong \sigma$. Hence, the assumption implies that for all $\sigma$:
$$ \mathcal C_{1, \sigma}= z_{d-2}(R_1^{\mathrm{cr}}(\wt(\sigma), \Eins\oplus \Eins)), \quad 
\mathcal C_{2, \sigma}= z_{d-2}(R_2^{\mathrm{cr}}(\wt(\sigma), \Eins\oplus \Eins)).$$

Since  $\Spec R_1^{\mathrm{cr}}(\wt(\sigma), \Eins\oplus \Eins))$ and $\Spec R_2^{\mathrm{cr}}(\wt(\sigma), \Eins\oplus \Eins))$ are contained in the reducible 
locus by \cite[Lem.3.5]{KW}, the assertion follows immediately from Theorem \ref{bm_cycle} and the fact that \eqref{cycle_map} preserves Hilbert--Samuel multiplicities.
\end{proof} 

\begin{remar} Let $R_1^{\square}$, $R_2^{\square}$, $R^{\square}$ be the framed deformation rings of $\rho_1$, $\rho_2$ and $\rho$
respectively, and let $R_1^{\square}(\wt, \tau)$,  $R_2^{\square}(\wt, \tau)$ and $R^{\square}(\wt, \tau)$ denote the quotients, which 
parameterize potentially semi-stable lifts of type $(\wt, \tau)$. It follows from \cite[Prop. 2.1]{KW} that $R_1^{\square}$ is formally smooth over $R_1$ of relative dimension $3$, 
$R_2^{\square}$ is formally smooth over $R_2$ of relative dimension $3$, $R^{\square}$ is formally smooth over $R^{\ver}$ of relative dimension $2$. Since these framing variables only keep track of the chosen basis, we deduce that $R_1^{\square}(\wt, \tau)$,  $R_2^{\square}(\wt, \tau)$ and $R^{\square}(\wt, \tau)$ are formally smooth over $R_1(\wt, \tau)$,  $R_2(\wt, \tau)$ and $R^{\ver}(\wt, \tau)$ of relative dimension $3$, $3$ and $2$ respectively. This allows to use  Theorem \ref{split_bm} to deduce an analogous  statement
for the framed deformations rings. Moreover, one may additionally consider potentially crystalline lifts and/or fix the determinant.
\end{remar}

\end{document}